\documentclass[11pt]{article}
\usepackage[a4paper,margin=2.4cm]{geometry}

\usepackage[utf8]{inputenc}
\usepackage[english]{babel}

\usepackage[colorinlistoftodos,prependcaption]{todonotes}

\usepackage{amsmath}
\usepackage{enumerate}
\usepackage{amsthm,amssymb}
\usepackage{amscd}
\usepackage{color}
\usepackage{hyperref}
\usepackage{fancyhdr}
\usepackage{bigints}

\usepackage{caption}

\usepackage{array}
\usepackage{makecell}

\usepackage{array}

\usepackage{scalerel}

 \newenvironment{msc}
 {\par \noindent\textit{2010 MSC:}\ \ignorespaces}
 {\par\vspace{3mm}}

\newenvironment{keywords}
 {\par \noindent\textit{Keywords:}\ \ignorespaces}
 {\par\vspace{3mm}} 

\newcommand{\lambdap}{
\lambda_{\scaleto{+}{4.5pt}}
}

\newcommand{\lambdam}{
\lambda_{\scaleto{-}{4.5pt}}
}

\newcommand{\lp}{l_{\scaleto{+}{4.5pt}}}
\newcommand{\lm}{l_{\scaleto{-}{4.5pt}}}
\newcommand{\lpp}{l_{\scaleto{+}{3.5pt}}}
\newcommand{\lmm}{l_{\scaleto{-}{3.5pt}}}

\newcommand{\itemEq}[1]{%
         \begingroup%
     \setlength{\abovedisplayskip}{-3pt}%
     \setlength{\belowdisplayskip}{-3pt}%
         \parbox[c]{\linewidth}{\begin{flalign}#1&&\end{flalign}}%
         \endgroup}

\newcommand{\itemEqNoLabel}[1]{%
         \begingroup%
     \setlength{\abovedisplayskip}{-3pt}%
     \setlength{\belowdisplayskip}{-3pt}%
         \parbox[c]{\linewidth}{\begin{flalign*}#1&&\end{flalign*}}%
         \endgroup}

\newtheorem{theorem}{Theorem}[section]
\newtheorem{proposition}[theorem]{Proposition}
\newtheorem{lemma}[theorem]{Lemma}

\newtheorem{remark}[theorem]{Remark}
\newtheorem{definition}[theorem]{Definition}

\title{Edge fluctuations for random normal matrix ensembles}

\author{David García-Zelada}
\date{}

\begin{document}

\maketitle

\begin{abstract}
A famous result going back to Eric Kostlan states that the moduli of the eigenvalues of random normal matrices with radial potential are independent yet non identically distributed. This phenomenon is at the heart of the asymptotic analysis of the edge, and leads in particular to the Gumbel fluctuation of the spectral radius when the potential is quadratic. In the present work, we show that a wide variety of laws of fluctuation are possible, beyond the already known cases, including for instance Gumbel and exponential laws at unusual speeds. We study the convergence in law of the spectral radius as well as the limiting point process at the edge. Our work can also be seen as the asymptotic analysis of the edge of two-dimensional determinantal Coulomb gases and the identification of the limiting kernels.
\end{abstract}

\begin{msc}
\small
60G55; 60F05; 60K35; 60G70
\end{msc}


\vspace{3mm}

\hrule

\section{Introduction}
We will be interested in the fluctuations
of the maxima of the moduli
of the eigenvalues of random normal matrices
confined by a radial potential. Equivalently,
these eigenvalues may be thought of as a
two-dimensional radial determinantal
Coulomb gas as defined
in \eqref{eq:CoulombGasDensity} below. 
The purpose of this article is two-fold.
On the one hand,
we want to 
present different classes
of universality 
by describing the subtleties
there may appear. In this way, we show
a variety of
possibilities emerging,
each possibility having different interesting
aspects. On the other hand, we want
to convey the simplicity of the methods,
which extend far
beyond the examples mentioned in this article. 
In particular, the local limiting point process
of any two-dimensional radial determinantal
Coulomb gas 
can be easily understood
as well as the
local limiting point process
of the norms. Furthermore,
the same methods also help one study the local 
limiting covariance kernel for random polynomials
such as the ones considered
in \cite{ButezGZ}.
I would like to also mention that
the regularities required for the potential $V$
are minimal and that this makes the proofs
even clearer by not looking
at the unimportant aspects of $V$.

Let us mention some of the previous
work on the fluctuations of the maxima.
The first result we are aware of
is the Gumbel fluctuations
at speed $\sqrt{n \log n}$
obtained by Rider \cite{Rider}
for the farthest particle of the Ginibre ensemble.
A generalization of this result
can be found in the work of Chafai and 
Péché \cite{ChafaiPeche} 
as well as in the recent work of
Ameur, Kang and Seo 
\cite{Ameur}
with the same speed 
$\sqrt{n \log n}$.
The hard-edge version has been
considered by Seo \cite{Seo}
who showed exponential fluctuations
at speed $n$.
On a series of articles,
\cite{JiangQi,GuiQi,ChangLiQi}, Qi and his
collaborators
have studied different
cases related to matrix models
which includes truncated circular unitary
matrices and products of matrices
from the spherical
and from the Ginibre ensemble. 
A class of potentials generated
by probability measures
has been recently considered by
Butez and the author in
\cite{ButezGZ}. Fluctuations
of the maxima for Coulomb gases have
also attracted physicists attention as we
can see, for instance, in
the work of Lacroix-A-Chez-Toine, Grabsch, Majumdar
and Schehr
\cite{Grabsch} where even an intermediate
deviation regime is explored. 
The farthest particle has also
been of interest for
fermionic systems
as in the work of Dean,
Le Doussal, Majumdar and Schehr \cite{Dean}.
Despite these efforts, the variety
of behaviors has not been really explored
and the methods used
up to this point have not been simple enough.

We now proceed to describe the model. 
Let $\mathcal N_n$
denote the set of $n$ by $n$ normal matrices
endowed with the measure $\mathrm d M$ 
induced by the restriction (to the regular part of
$\mathcal N_n$)
of the euclidean inner product 
$\langle M, N \rangle = \mathrm{Tr}(M N^*)$
on the space of $n$ by $n$ matrices.
We will be interested in the measure
\begin{equation}
\label{eq:NormalMatrixLaw}
e^{-Q(M)} \mathrm d M,
\end{equation}
for some continuous function 
$Q:\mathbb C \to \mathbb R$.
By \cite{NormalMatrixPhysics,NormalMatrices},
if \eqref{eq:NormalMatrixLaw}
has finite measure and if
 $M$ follows a law proportional to \eqref{eq:NormalMatrixLaw}, the 
eigenvalues of $M$ (symmetrically ordered)
follow a law proportional to the measure given by
$\prod_{i<j}^n |x_i - x_j|^2 
e^{-\sum_{i=1}^n Q(x_i)}
{\rm d}\ell_{ \mathbb C^n}(x_1,\dots,x_n),$
where $\ell_{\mathbb C^n}$
is the Lebesgue measure on $\mathbb C^n$.
This measure has the following interesting
interpretation in classical electrostatics.
Since we are interested in the case where $Q$
is radial,
we fix a continuous function
$V: [0,\infty) \to \mathbb R$
and consider
a system of $n$
particles in $\mathbb C$
of the same charge $q$ and confined
by the potential $V$. 
This system, known as a Coulomb gas, has the energy
\[H(x_1,\dots,x_n)
= - q^2 \sum_{i<j}^n\log|x_i - x_j| 
+    q  \sum_{i=1}^n V(|x_i|) \]
and follows
the Gibbs probability measure 
at inverse temperature $\beta>0$ and energy $H$, i.e.
the probability measure
proportional to $e^{-\beta H}
\mathrm d \ell_{\mathbb C^n}$.
The case related to random normal matrices
is the one where
$\beta q^2 = 2$ since
in this case the 
Gibbs probability measure 
is given
by
\begin{equation}
\label{eq:CoulombGasDensity}
\mathrm d \mathbb P_n^\kappa
(x_1,\dots,x_n) =\frac{1}{\mathcal Z}
 \prod_{i<j}^n |x_i - x_j|^2 
e^{-2\kappa \sum_{i=1}^n V(|x_i|)}
{\rm d}\ell_{ \mathbb C^n}(x_1,\dots,x_n)
\end{equation}
where $\kappa = 1/q$ and where $\mathcal Z$
is a normalization constant.
Then, the eigenvalues
of a random normal matrix 
that follows a law proportional to
\eqref{eq:NormalMatrixLaw} 
with $Q(z) = 2\kappa V(|z|)$ can be thought of
as a system of particles in $\mathbb C$ 
of charge $1/\kappa$
and at inverse temperature $\beta = 2\kappa^2$.
For
$\mathbb P_n^\kappa$
to be well-defined
(i.e., $\mathcal Z < \infty$)
we shall assume that
 $\kappa > n$ and that
\[\liminf_{r \to \infty} 
\left\{V(r) - \log r \right\} > - \infty.\]
To relax
the condition $\kappa > n$
we can assume that
the potential is 
\emph{strongly confining}, i.e.,
\begin{equation}
\label{eq:StronglyConfiningCondition}
\liminf_{r \to \infty} 
\frac{V(r)}{\log r} >1
\end{equation}
or, equivalently,
$\liminf_{r \to \infty} 
\left\{V(r) - Q\log r \right\} >
-\infty$ for some 
$Q>1$.
If \eqref{eq:StronglyConfiningCondition}
is not satisfied we say that the potential
is \emph{weakly confining}.
We shall also consider the
degenerate case where $V = \infty$
outside a disk which, for simplicity, 
we take it to be the 
closed unit disk centered at zero
$\overline {\mathbb D}$.
These are examples of
\emph{hard-edge} systems
and they are determined by a
continuous function
$V:[0,1] \to \mathbb R$ 
and a positive number $\kappa$. The Gibbs 
probability measure would be
\begin{equation}
\label{eq:CoulombGasDensityHardEdge}
\frac{1}{\tilde{\mathcal Z}}
 \prod_{i<j}^n |x_i - x_j|^2 
e^{-2\kappa \sum_{i=1}^n V(|x_i|)}
{\rm d}\ell_{ \overline {\mathbb D}}(x_1)\dots
{\rm d}\ell_{ \overline {\mathbb D}}(x_n),
\end{equation}
where $\ell_{\overline {\mathbb D}}$ denotes
the Lebesgue measure restricted to
$\overline {\mathbb D}$. It may be thought of
as a particular case
of \eqref{eq:CoulombGasDensity}
where we let $V(r) = \infty$ for
$r>1$.

By well-known large deviation
principles
(see, for instance, 
\cite[Theorem 1.2]{ChafaiJung}
for an idea of the proof), if $(x_1^{(n)},
\dots,x_n^{(n)}) \sim 
\mathbb P_n^{\kappa_n}$
and $\kappa_n/n 
\xrightarrow[n \to \infty]{} 1$
(and under some conditions on $\kappa_n$
in the weakly confining case)
we have that
\[\frac{1}{n}\sum_{i=1}^n \delta_{x_i^{(n)}}
\xrightarrow[n \to \infty]{\mathrm{a.s.}} \mu_V\]
where $\mu_V$ is the probability
measure
that minimizes
the functional
\[\nu \mapsto 
\int_{\mathbb C \times \mathbb C}
\left(\log \frac{1}{|x-y|}
+ V(|x|) + V(|y|) \right)
\mathrm d \nu(x)
\, \mathrm d \nu(y).\]
Finally,
let us say some words about the
laws \eqref{eq:CoulombGasDensity} and 
\eqref{eq:CoulombGasDensityHardEdge}.
The exponent $2$ in $\prod_{i<j}^n |x_i - x_j|^2$
makes the system quite feasible to study. It 
enjoys the property
of being a \textit{determinantal point process},
which will be further explained
in Section \ref{sec:determinantal}.
This structure 
allows us to find limits of point processes
by studying the limit of a function of
two variables. Some examples of limits are given
in Subsection \ref{subAp:KernelGumbel}
and in Proposition \ref{prop:HardEdgeKernelLimit}.
Furthermore,
in the study of the moduli
$(|x_1^{(n)}|,\dots,|x_n^{(n)}|)$,
an idea due to Kostlan \cite[Lemma 1.4]{Kostlan}
shows that, if 
$Y_0^{(n)},\dots,Y_{n-1}^{(n)}$
are independent random variables
such that $Y_k^{(n)}$ has a density proportional to
$r^{2k+1}e^{-2\kappa V(r) }$, then
\[\{|x_1^{(n)}|,\dots,|x_n^{(n)}|\}
\sim \{Y_0^{(n)},\dots,Y_{n-1}^{(n)}\}\]
as point processes on $[0,\infty)$. 
A more general statement will
be recalled in Subsection \ref{sub:Kostlan}.
Interestingly, the asymptotic analysis
of the moduli then becomes
an asymptotic analysis of independent and non-identically distributed random variables.


Let us describe the content of each
section.

In Section \ref{sec:FarthestParticles} we present the main results.
We consider
 potentials whose equilibrium measures
are compactly supported and
we show that 
different behaviors 
for the farthest particles may arise.
In Subsection \ref{sub:Coulombian} 
we consider potentials
for which the farthest particle
keeps the Coulombian behavior. This means
that the limit point process at the scale considered
is not Poissonian and is still related
to a determinantal point process.
The potentials involved
may be either strongly confining
or weakly confining, i.e., either
\eqref{eq:StronglyConfiningCondition}
is true or not.
 In Subsection \ref{sub:Gumbel}
we consider potentials for which a Gumbel
fluctuation appears.
One of the cases shown is
a class of strongly confining potentials
and the other one is a class
of weakly confining potentials
with the parameter $\kappa$ 
growing fast to infinity.
In Subsection \ref{sub:Exponential}
we consider
a hard-edge potential for which
exponential fluctuations at speed $n^2$ 
are obtained. 

In Section \ref{sec:determinantal}
we give a short introduction to the
determinantal structure we need
together with the property
that allows us to describe the 
cumulative distribution function
of the maxima of the moduli in an
explicit way.
It is
the generalization of a
property discovered by Kostlan
and rediscovered, for instance, by
Fyodorov and Mehlig
\cite[Equation 16]{Fyodorov}.

The proofs are divided in three sections.
Section
\ref{sec:Minima}, \ref{sec:ProofsGumbel} and
\ref{sec:ProofExponential} contain
the proofs
stated in Subsection \ref{sub:Coulombian},
 \ref{sub:Gumbel} and 
 \ref{sub:Exponential} respectively.

Finally, Section 
\ref{sec:Appendices} contains four
appendices with further information the reader
might find interesting.
Subsection \ref{subAp:RandomNumber}
contains
a version of
Theorem \ref{th:MaxFiniteParticles}
that is longer to state
and where a 
random limiting number of particles appears. 
Subsection \ref{subAp:Hard-edgeAlphaCase}
contains the $\alpha>1$ version of Theorem
\ref{th:HardEdge}. 
Subsection \ref{subAp:KernelGumbel}
contains the limiting kernel at the edge
of some systems described
in Subsection \ref{sub:Gumbel} together
with a remark about the relation between
the limiting kernel and the Gumbel fluctuations.
Subsection \ref{subAp:GumbelConnectionWeakly}
explains an interesting connection between
the Gumbel distribution and the limit
of the maxima from
Theorem \ref{th:MaxAnnulus}.

\section{Behaviors of the farthest
particles}
\label{sec:FarthestParticles}
In this article,
we will be interested in the
case where
$\mu_V$ is compactly supported.
Without loss of generality,
 we can assume that
\begin{equation}
\label{eq:SupportCondition}
\partial \mathbb D \subset 
\mathrm{supp}\, \mu_V
\subset \overline{ \mathbb D},
\end{equation}
where $\mathbb D$ is the open
unit disk centered at zero.
Recall we suppose that
 $V:[0,\infty) \to \mathbb R$,
or $V:[0,1] \to \mathbb R$ in the hard-edge case,
is continuous. Using
 Frostman's conditions 
(see
\cite{Saff}), the condition in
\eqref{eq:SupportCondition} 
traduces,
by adding a constant to $V$,
in the following properties
for the potential.
\begin{definition}[Standard properties]
\label{def:StandardProperties}
A continuous function
$V:[0,\infty) \to \mathbb R$ or
$V:[0,1] \to \mathbb R$ is said to satisfy
the \emph{standard properties} if
\begin{itemize}
\item $V(r) > \log r
\quad \mbox{ when } r < 1, $
\item $V(r) \geq \log r
\quad \mbox{ when } r > 1, $ and
\item $V(1) = 0$.
\end{itemize}
\end{definition}
Every $V$ is assumed to satisfy these properties
throughout the article. Notice that the second condition
is vacuous for $V:[0,1] \to \mathbb R$,
or we may say it is always true
if we define $V(r)=\infty$ for $r>1$.
Although
this second condition will always be implied
by the other conditions in the theorems, it
is important to keep it in mind. Given
a sequence $\{\kappa_n\}_{n\geq 1} $
we consider a random element
$(x_1^{(n)},\dots,x_n^{(n)})
\sim \mathbb P_n^{\kappa_n}$, where
$\mathbb P_n^{\kappa_n}$ is defined by
\eqref{eq:CoulombGasDensity} or \eqref{eq:CoulombGasDensityHardEdge},
and study 
\[M_n = \max \left(|x_1^{(n)}|,
\dots,|x_n^{(n)}| \right). \]
More precisely,
we look for two sequences of real numbers
$\{\alpha_n\}_{n \geq 1}$
and $\{\beta_n\}_{n \geq 1}$ such that
\[\alpha_n M_n +
\beta_n  \quad 
\mbox{converges in law}\]
to a non-deterministic 
random variable. 
We will see that the behavior
of $M_n$
depends on the behavior of 
$V$ with respect to $\log r$.
It will also be seen to depend on the asymptotic
behavior of 
$\kappa_n$ or, equivalently, on
the behavior of the charge
of a single particle $1/\kappa_n$.
This dependence will
be explained by
five examples. 

\begin{itemize}
\item In Theorem \ref{th:MaxAnnulus},
$M_n$ converges to a random variable
supported in a closed interval
inside $[1,\infty)$. This interval
depends on the behavior of $V$
in $[1,\infty)$ but is otherwise arbitrary. 
For instance, the interval could be bounded away
from $1$ in which case an infinite number of particles
stay far from the support of $\mu_V$.
Moreover, the potential can be strongly
confining, in the sense of
\eqref{eq:StronglyConfiningCondition},
while the farthest particles stay 
living in an annulus.
\item In Theorem
\ref{th:MaxFiniteParticles}, $M_n$ converges
to infinity. After a proper rescaling,
$M_n$ is seen to converge to a non-deterministic random
variable. Interestingly,
at the same scale there are only
a finite number of particles that accompany 
the one of modulus $M_n$.
	\item Theorem \ref{th:Strongly} shows 
Gumbel fluctuations
arising when the potential is strongly confining.
This is stated for a family
of behaviors of $V(r)$ near $r=1$.
There is a compatibility,
stated in Remark \ref{rem:GumbelAndEdge},
between the Gumbel fluctuations and
the limiting
point process at the edge
described in Proposition
\ref{prop:StronglyKernelLimit}.
	\item Theorem \ref{th:WeaklyKappa} shows
Gumbel fluctuations
when the potential is still weakly confining
but $\kappa_n - n \xrightarrow[]{} \infty$.
In this case, there is no longer
a compatibility
between the Gumbel fluctuations
and the limiting point process
at the edge, described in
Proposition \ref{prop:WeaklyKernelLimit}.

\item 
Theorem \ref{th:HardEdge}
treats the hard-edge case.
The parameter for
the exponential fluctuations
are shown
to depend on the potential
near the edge but, in the particular family
of potentials described,
it also depends on what happens inside of
the unit disk.
These exponential fluctuations
are obtained by using the limiting kernel
at the edge from
Proposition \ref{prop:HardEdgeKernelLimit}.

\end{itemize}

There are two further 
complementary classes of potentials described
in Theorem \ref{th:MaxFiniteParticlesRandom}
and Theorem \ref{th:HardEdgeAlpha} in the appendix.
To complement the previous explanation, 
we would like to give some examples
implied by the theorems.
Unless otherwise stated,
$V(r) > \log r$ for $r \neq 1$. 

We begin by giving some
non-Poissonian examples. We call them 
Coulombian examples because
they are related
to some limits of determinantal point processes. 
These limiting point processes
may be thought of as instances of
Coulomb gases
with an infinite number of particles.
In the left-most column of Table \ref{tab:Coulomb}
some conditions on the potential are stated.
In the second column, 
speed means the coefficient $\alpha_n$
necessary to obtain the convergence in law 
of $\alpha_n M_n$. Finally,
the right-most column tells us
the cumulative distribution function of the
limit of $\alpha_n M_n$. For simplicity,
we have chosen $\kappa_n = n + \chi$ as the inverse charge
for $\chi > 0$.
{\small
\\
$\left. \right.$
\\
\begin{minipage}{\linewidth}
\renewcommand{\arraystretch}{1.3}
\begin{center} \captionsetup{type=figure}
\captionof{table}{Coulombian examples} \label{tab:Coulomb} 
\begin{tabular}{|m{15em}||c|c|}
		\hline
		Conditions & Speed & Cumulative distribution
		function \\
		\hline \hline
		\thead{$\left. \right.$ \\$V(r) = \log r$ 
		only on $[R,\infty)$\\
		$\left. \right.$ }  & 1 & 
		$\displaystyle\prod_{k=0}^\infty 
	\left( 
1 - \left(\frac{R}{t}\right)^{2k+2\chi}
	 \right)$\\
	 \hline
		\thead{$\left. \right.$ \\ $V(r) = \log r$ 
		only on $[1,R]$\\
		$\left. \right.$ } & $1$ & 
		$\displaystyle
		\prod_{k=0}^\infty 
		\left(\frac{1- t^{-2k-2\chi} }{1-
		R^{-2k-2\chi} } \right)$
		\\
		\hline
		\thead{$V(r) = \log r
		+ \frac{1}{2} r^{-\alpha}
		+ o(r^{-\alpha}) $,\\
		$V(r) = \log r + 
		|1-r| + o(|1-r|)$,	\\
		$\ell : =
		\alpha/2 - \chi  \notin \mathbb Z$} 
		& $n^{-1/\alpha}$ & $\displaystyle \prod_{k=0}^
				{\lfloor \ell
				\rfloor}
	\frac{\Gamma
	 \left( 
	 \frac{2(k+\chi)}{\alpha},
	  t^{-\alpha} \right)
	}
	{\Gamma
	 \left( 
	 \frac{2(k+\chi)}{\alpha} 
	\right)}$ \\
		\hline 
		\thead{$V(r) = \log r
		+ \frac{1}{2} r^{-\alpha}
		+ o(r^{-\alpha}) $, \\
		$V(r) = \log r + 
		|1-r| + o(|1-r|)$,	\\
		$\ell:=\alpha/2 - \chi  \in \mathbb Z$} 
		& $n^{-1/\alpha}$ & $\displaystyle
	\left(
	e^{- t^{-\alpha}}\! \!
	+
	(1-e^{- t^{-\alpha}})
	\frac{\alpha}{1+\alpha}
	\right)
				\prod_{k=0}^
				{ \ell
				- 1
				}
	\frac{\Gamma
	 \left( 
	 \frac{2(k+\chi)}{\alpha},
	 t^{-\alpha} \right)
	}
	{\Gamma
	 \left( 
	 \frac{2(k+\chi)}{\alpha} 
	\right)} $ \\
		\hline 
	\end{tabular}\par
\end{center}
\end{minipage}
\\ 
$\left. \right.$
\\}
The first and second examples are particular
cases of Theorem \ref{th:MaxAnnulus}.
The third example is a consequence
of Theorem \ref{th:MaxFiniteParticles}
while the fourth example is a consequence of
the random number of particles version stated
in
Theorem \ref{th:MaxFiniteParticlesRandom}
in the appendix.

Next, we give two examples of Gumbel fluctuations.
More precisely, in those cases there exist
$\alpha_n$ and $\beta_n$ such that
$\alpha_n M_n + \beta_n$
converges to a random variable
whose cumulative distribution function
is
$t \in \mathbb R \mapsto e^{-e^{-t}}$.
The left-most column 
of Table \ref{tab:Gumbel} contains the conditions on
the potential,
the second column contains the conditions on 
the inverse charge $\kappa_n$,
and the right-most column contains
the speed $\alpha_n$.
{\small
\\
$\left. \right.$
\\
\begin{minipage}{\linewidth}
\renewcommand{\arraystretch}{1.3}
\begin{center} \captionsetup{type=figure}
\captionof{table}{Gumbel examples} \label{tab:Gumbel}
	\begin{tabular}{|m{22em}||c|c|}
		\hline
		Conditions & Inverse charge & Speed \\
		\hline 
		\hline 
		\thead{$V(r)=\log r +
		\frac{1}{2\alpha} |1-r|^\alpha
		+ o(|1-r|)^{\alpha+\varepsilon} $ \\}
		& $\kappa_n
		= n + o\left(
		\frac{n}{\log n} \right)^{1/\alpha}$
		& $n^{1/\alpha}(\log n)^{1-1/\alpha}$ \\
	\hline
	\thead{$V(r)= \log r +
		\frac{1}{2\alpha} (1-r)^\alpha
		+ o(1-r)^{\alpha} $ 
		as $r \to 1^{-}$ \\ and 
		$V(r) = \log r$ if $r \geq 1$}
		& $\kappa_n - n \to \infty$
		& $2(\kappa_n -n )$ \\
	\hline
	\end{tabular}\par
\end{center}
\end{minipage}
\\ 
$\left. \right.$
\\}
The first example is a particular case
of Theorem \ref{th:Strongly}
while the second example
 is implied by
Theorem \ref{th:WeaklyKappa}.

Finally, we give two examples of exponential 
fluctuations. More precisely,
in those cases
there exists $\alpha_n$ such that $\alpha_n (1-M_n)$
converges to a random variable
whose cumulative distribution function
is
$t \in [0,\infty) \mapsto 1-e^{-at}$ for some $a>0$.
The left-most column of
Table \ref{tab:Expo} contains the conditions on
the potential,
the second column contains the conditions on 
the inverse charge $\kappa_n$,
and the right-most column tells us if the parameter
$a$
of the exponential
 depends only on the conditions of the potential
stated on the first column or not.
\\
$\left. \right.$
\\ {\small
\begin{minipage}{\linewidth}
\renewcommand{\arraystretch}{1.3}
\begin{center} \captionsetup{type=figure}
\captionof{table}{Exponential examples} \label{tab:Expo}
	\begin{tabular}{|m{21.7em}||c|c|c|}
		\hline
		Conditions & Inverse charge & 
		Speed & Parameter \\
	\hline
	\thead{$V(r)=\log r +  (1-r) + o(1-r)$ 
		as $r \to 1^{-}$		\\
		and $V(r) = \infty$ if $r > 1$}
		& No condition
		& $n^2$ & 
	\thead{Depends on  \\ $V$ 
	inside $[0,1)$}
		\\
	\hline
	\thead{$V(r)=\log r +  
	\frac{1}{2\alpha}(1-r)^\alpha + o(1-r)$ 
		as $r \to 1^{-}$		\\
		and $V(r) = \infty$ if $r > 1$}
		& $\kappa_n = n + o(n^{1/\alpha})$
		& $n^{2/\alpha}$ & Universal \\
	\hline
	\end{tabular}\par
\end{center}
\end{minipage}
\\ 
$\left. \right.$
\\}
Recall that $V(r) = \infty$ if $r>1$
means we are dealing with
the hard-edge case and $V$ is actually
a continuous function on $[0,1]$.
The first example of Table \ref{tab:Expo} 
can be found in
Theorem \ref{th:HardEdge}
while the second example is found in
Theorem \ref{th:HardEdgeAlpha}
in the appendix.

\begin{remark} [Equivalent statement]
In every theorem we will 
have some
 $R \in [-\infty,\infty)$
and some $\alpha_n$ and $\beta_n$ such that
$\{\alpha_n|x_1^{(n)}| + \beta_n,\dots,
\alpha_n|x_n^{(n)}|+\beta_n\} \cap (R,\infty)$ 
converges to some
point process $\mathcal X$
together with the convergence of 
 $\alpha_n M_n + \beta_n$
towards 
the maximum of $\mathcal X$. Nevertheless, both
together imply and are implied by 
the following statement. For every 
$f: \mathbb R \to  \mathbb R$
whose support is contained in
$(R,\infty)$,
 \[
 \sum_{i=1}^n f(\alpha_n |x_i^{(n)}| + \beta_n)
 \xrightarrow[n \to \infty]{\mathrm{law}}
 \sum_{x \in \mathcal X} f(x).\]
\end{remark}

\subsection{Coulombian behavior}
\label{sub:Coulombian}

In both of the theorems presented here 
$V(r) - \log r$ touches zero 
away from $r = 1$ and in a non trivial way.
Their proofs, given in Section
\ref{sec:Minima}, are motivated
by the methods used in \cite{ButezGZ}.
Even 
though the proofs will give
the behavior of the point process
$\{x_1^{(n)},\dots,x_n^{(n)}\}$,
the theorems will only be stated
for the point process of the moduli
$\{|x_1^{(n)}|,\dots,|x_n^{(n)}|\}$.
The behavior of the point process
$\{x_1^{(n)},\dots,x_n^{(n)}\}$
may be obtained from
Proposition \ref{propC:MinA}
and Proposition \ref{propC:MinFiniteOne}.

The main interesting aspect
of the following theorem is that,
despite the fact that the support
of $\mu_V$
is contained in the closed unit disk,
the limit of the particle 
farthest from the origin may
live in an annulus at a finite distance
from the unit disk.
Moreover, in the limit, there will be
an infinite number of particles
living in such complement
and they will only accumulate
at the inner boundary of the annulus.

Notice that the following result
involves a potential
that does not need to be
weakly confining. 
It is an example
where the fact of being
strongly confining, in our sense, 
does not immediately imply
a Gumbel fluctuation of the 
farthest particle.

\begin{theorem}[Particles in an annulus]
\label{th:MaxAnnulus}
Suppose $V$ satisfies the 
\emph{standard properties} and
\begin{itemize}
\item \itemEqNoLabel{V (r) > \log r\, \mbox{ for every } r \in (1,R) \cup (\widetilde R,\infty) \mbox{ and }}
\item \itemEqNoLabel{V(r) = \log r\,
\mbox{ for every }
r \in [R,\widetilde R),}
\end{itemize}
for some $R \in [1,\infty)$ and 
$\widetilde R \in (1,\infty]$
such that $R < \widetilde R$.
Suppose that,
for some $\chi>0$,
\[\kappa_n =  n +
\chi + o(1).\]
Consider a sequence
$\{Y_k\}_{k \geq 0}$
of independent random variables
taking values in $[R,\widetilde R)$
such that $Y_k$ has a density proportional to
$
r \in [R,\widetilde R) \mapsto 
r^{-2(k + \chi)-1}$ .
Then, as point processes on 
$(R,\infty)$,
\[
\{|x_i^{(n)}|: 1 \leq i \leq n\} \cap 
(R,\infty)
\xrightarrow[n \to \infty]{
\mathrm{law}}\{Y_k: k \geq 0\}.\]
Furthermore, the maximum of
$|x_1^{(n)}|,\dots,
|x_n^{(n)}|$ converges in law
to the maximum of 
$\{Y_k\}_{k \geq 0}$.
More explicitly, for  every $t \in [R,\widetilde R)$,
\[\lim_{n \to \infty}\mathbb P
	\left(M_n
	\leq t
	\right)
	=
	\prod_{k=0}^\infty 
	\left( 
	\frac{R^{-2(k+\chi)}-t^{-2(k+\chi)} }
	{R^{-2(k+\chi)}-\widetilde R^{-2(k+\chi)}}
	 \right).\]

\end{theorem}

The limiting point process 
for $\widetilde R < \infty$
can be thought of as a conditioned version of the
limiting point process for $\widetilde R =\infty$
and, by a scaling, it can be taken
to the limiting point process for $R=1$.
For $\chi \in (0,1]$, the latter is related
to a weighted Bergman kernel of $\mathbb C 
\setminus \overline{\mathbb D}$ while
for $\chi > 1$ it would be related to
a truncated version of it as can be
seen from the results of 
Proposition \ref{propC:MinA}. We may see
\cite{BergmanSpaces} for more information on
Bergman kernels.
In Proposition
\ref{prop:GumbelWeakly} from the appendix
we will describe
the behavior
as $\chi$ goes to infinity
of the limit of the maxima
obtained
in Theorem \ref{th:MaxAnnulus}.

The main aspect of
the next theorem is
that we will only
see a finite number
of particles at the same
scale of the particle farthest
from the origin. There is a version
where the number of particles is random
and is
stated in Theorem \ref{th:MaxFiniteParticlesRandom}
for convenience of the reader.
The conditions satisfied
by $V$ imply
that
$\mu_V(\mathbb D)< 1$ so that
$\mu_V$
must give non-zero charge
to $\partial \mathbb D$ but it may
still be charged near $\partial \mathbb D$.

\begin{theorem}[Finite limiting point process]
\label{th:MaxFiniteParticles}

Suppose $V$ satisfies the \emph{standard properties} and 
\begin{itemize}
\item 
	\itemEqNoLabel{V(r) > \log r \,
	\mbox{ for every } 
	r > 1,}
\item
	\itemEqNoLabel{\lim_{r \to \infty}
 	r^\alpha \left( V(r) - \log r \right) = 
 	\gamma,}
\item 
	\itemEqNoLabel{\lim_{r \to 1^+} \frac{V(r)}{r-1}
	 >1 \mbox{ and }
	  \lim_{r \to 1^-} \frac{V(r)}{r-1}
	 < 1,}
\end{itemize}
for some
positive numbers
$\alpha, \gamma >  0$. Suppose that, for some 
$\chi \in (0,\alpha/2) $ such that
$\alpha/2 - \chi \notin \mathbb Z$,
\[\kappa_n = n + \chi + o(1).\]
Consider a sequence
$\{Y_k\}_{k \geq 0}$
of positive independent random variables
such that $Y_k$ has a density proportional to
$r \in (0,\infty) \mapsto
r^{-2(k + \chi)-1}
e^{-r^{-\alpha}}$.
Then, as point processes on 
$(0,\infty)$,
\[
\{n^{-1/\alpha}|x_i^{(n)}|: 1\leq i\leq n\}
\xrightarrow[n \to \infty]{\mathrm{law}}
\{(2\gamma)^{1/\alpha}Y_k: 0 \leq k 
 \leq \lfloor \alpha/2 - \chi
 \rfloor \}.\]
Furthermore, 
$n^{-1/\alpha}\max \{
|x_1^{(n)}|
,\dots, |x_n^{(n)}|\}$
converges in law to
the maximum of 
$\{Y_k\}_{0\leq k \leq 
\lfloor \alpha/2 - \chi
 \rfloor}$.
More explicitly, for every $t > 0$,
	\[\lim_{n \to \infty}\mathbb P
	\left(n^{-1/\alpha}
	M_n
	\leq t
	\right)
	=
				\prod_{k=0}^
				{\lfloor \alpha/2-
				\chi 
				\rfloor}
	\frac{\Gamma
	 \left( 
	 \frac{2(k+\chi)}{\alpha},
	 2\gamma t^{-\alpha} \right)
	}
	{\Gamma
	 \left( 
	 \frac{2(k+\chi)}{\alpha} 
	\right)},
	\]
where $\Gamma$ with two parameters
is the upper incomplete
gamma function.	
\end{theorem}

If $\chi = 1$, the limit point process
obtained in Theorem \ref{th:MaxFiniteParticles}
is related to a system of particles in $\mathbb C$
with density proportional to
$ \prod_{i<j}^n|x_i - x_j|
e^{-\sum_{i=1}^n |x_i|^\alpha }$.
For general $\chi > 0$ the potential 
has an additional term.
See Proposition \ref{propC:MinFiniteOne}.


\subsection{Gumbel behavior}
\label{sub:Gumbel}

Here we consider two theorems
where Gumbel fluctuations 
of the maxima of the moduli appear.
They are proved in Section \ref{sec:ProofsGumbel}.
The first theorem treats potentials
that are strongly confining including
the quadratic potential
treated in 
 \cite{Rider} as a particular case.
The second theorem treats what we have
called weakly confining potentials
but where $\kappa_n - n \to \infty$ so
that 
Theorem \ref{th:MaxAnnulus} no longer applies.
At the end of this section, we give some intuition on
the conditions of the theorems 
by using potentials generated by 
radial positive measures.

There is an interesting connection between the
Gumbel fluctuations
and the local limit of these Coulomb gases
at the unit circle that only holds for 
$\xi=0$ from Theorem \ref{th:Strongly}. 
Since it is not essential to obtain 
the Gumbel fluctuations, we
have stated this in Subsection \ref{subAp:KernelGumbel}
in the appendix.
By somewhat simpler calculations than the ones
used
to obtain the Gumbel fluctuations,
we obtain the limiting kernel at the edge
(Proposition \ref{prop:StronglyKernelLimit}
and Proposition \ref{prop:WeaklyKernelLimit})
and in Remark \ref{rem:GumbelAndEdge}
we mention its connection with the Gumbel fluctuations
when there is one.

%

%
%
%

\begin{theorem}[Gumbel fluctuations for strongly confining potentials]
\label{th:Strongly}
Suppose $V$ satisfies the \emph{standard properties} and
\begin{itemize}
\item 
	\itemEq{V(r) = \log r
	+ \frac{\lambdap}{\alpha}
	(r-1)^\alpha
	+o(r-1)^{\alpha+\varepsilon}
	\mbox{ as } r \to 1^+,
						\label{eq:2PotentialEpsilon}} 
\item 
	\itemEq{\label{eq:LeftPotential}
	V(r) = \log r
	+ \frac{\lambdam}{\alpha}(1-r)^\alpha
	+o(1-r)^{\alpha}
	\mbox{ as } r \to 1^- ,}
\item \itemEq{V(r) > \log r
\, \mbox{ for every }
r > 1 \mbox{ and } \nonumber} 
\item 
   \itemEq{\label{eq:2StronglyConfiningInTheorem}
	\liminf_{r \to \infty} 
	\frac{V(r)}{\log r} >1
	\quad \mbox{(strongly confining),}}
\end{itemize}
for some positive numbers
$\varepsilon,\lambda_{\scaleto{+}{4.5pt}},
\lambda_{\scaleto{-}{4.5pt}} >0$ and $\alpha \geq 1$.
Suppose that, for some 
$\xi \in \mathbb R$,
\[
\kappa_n = n
+ \xi \left(\frac{2
\lambdap n}{\log n} 
\right)^{1/\alpha}
+ o \left(\frac{n}{\log n} 
\right)^{1/\alpha}.\]
Let $\Delta_n$ be the unique solution to
 $e^{\Delta_n/\alpha}
{\Delta}_n = n^{1/\alpha}$ 
and define
\[\delta_n = 
 \left(
\frac{\Delta_n}{
2\lambda_{\scaleto{+}{4.5pt}}
\kappa_n}
 \right)^{1/\alpha} 
 \quad \mbox{ and } \quad 
A =
\left(\frac{2}{\alpha}\right)^{1-1/\alpha}
\Gamma\left(
\frac{1}{\alpha} \right)
\left(\frac{1}{\lambda_{\scaleto{-}{4.5pt}}^{1/\alpha}}
+
\frac{1}{\lambda_{\scaleto{+}{4.5pt}}^{1/\alpha}}
\right).\]
Then, as point processes on $\mathbb R$,
\[
\{
\delta_n^{-1}\Delta_n
\big(|x_i^{(n)}| - 1 - 
\delta_n \big): 1 \leq i \leq n\}
\xrightarrow[n \to \infty]{\mathrm{law}}
\mathcal P,
\]
where $\mathcal P$ is a Poisson point process
on $\mathbb R$
with intensity $A^{-1} e^{-2\xi} e^{-s}$.
Furthermore,
\[
\lim_{n \to \infty}
\mathbb P\left(\delta_n^{-1}\Delta_n
\left(M_n - 1 - 
\delta_n \right)
\leq t
\right)
=
\exp\left(-A^{-1} e^{- t -2 \xi }\right)  \]
for every $t \in \mathbb R$.
\end{theorem}

The $\varepsilon$
in condition \eqref{eq:2PotentialEpsilon}
is essential, i.e., the theorem
would not hold if we take 
$\varepsilon = 0$ in 
\eqref{eq:2PotentialEpsilon}. In that
case, we
would have to look at more details about
the behavior of $V$ near $1$.

Let us say a few words about the coefficients.
We can replace $\Delta_n$ by
any sequence $\{ \Delta_n\}_{n \geq 1}$
that satisfies
\begin{equation}
\label{eq:DeltaDefinitionAsympt}
e^{\Delta_n/\alpha}
{\Delta}_n \sim n^{1/\alpha}.
\end{equation}
A simpler version of $\Delta_n$ would be
\[\Delta_n = 
\log n - \alpha \log \log n\]
but  we preferred to use the solution of
$e^{\Delta_n/\alpha}
{\Delta}_n = n^{1/\alpha}$
to emphasize the importance of
 \eqref{eq:DeltaDefinitionAsympt}. Moreover,
replacing $\delta_n^{-1}\Delta_n$
by
$(2\lambda_{\scaleto{+}{4.5pt}})^{1/\alpha}
n^{1/\alpha} (\log n)^{1-1/\alpha}$ in the multiplicative
coefficient we have, for instance,
\[
\lim_{n \to \infty}
\mathbb P\left((2\lambda_{\scaleto{+}{4.5pt}})^{1/\alpha}
n^{1/\alpha} (\log n)^{1-1/\alpha}
\left(M_n - 1 - 
\delta_n \right)
\leq t
\right)
=
\exp\left(-A^{-1} e^{- t -2 \xi }\right) . \]
Nevertheless, we have chosen
to write $\delta_n^{-1}\Delta_n$ for simplicity
of notation. Finally, if $\alpha > 1$ we could
have defined
$\delta_n = \left(
\frac{\Delta_n}{
2\lambda_{\scaleto{+}{4.5pt}}
n}
 \right)^{1/\alpha}$ instead of 
 $ \left(
\frac{\Delta_n}{
2\lambda_{\scaleto{+}{4.5pt}}
\kappa_n}
 \right)^{1/\alpha}$ due to the conditions
 on $\kappa_n$.

The interest of the following
theorem is that it involves 
very weakly confining potentials
so that the coefficients
cannot be guessed from the
limiting kernel at the edge.
See
Proposition \ref{prop:WeaklyKernelLimit}.
and Remark \ref{rem:GumbelAndEdge} in the appendix.

\begin{theorem}[Gumbel fluctuations for weakly
confining potentials]
\label{th:WeaklyKappa}
Suppose $V$ satisfies the \emph{standard properties} and
\begin{itemize}
\item
	\itemEqNoLabel{V(r) = \log r
	+ \frac{\lambda}{\alpha}(1-r)^{\alpha}
	+o(1-r)^{\alpha} 
\mbox{ as } r \to 1^- \mbox{ and }}
\item
	\itemEqNoLabel{V(r) = \log r 
	\,\mbox{ for every } r\geq 1,}
\end{itemize}
for some
$\lambda > 0$ and $\alpha \geq 1$.
Suppose that $\kappa_n>n$
satisfies $\kappa_n/n \to 1$ and that
\[
\lim_{n \to \infty}
\left(\kappa_n - n\right) = \infty
\quad
\mbox{and}
\quad 
\lim_{n \to \infty}
\frac{\kappa_n - n}{
n^{1/\alpha}}
= c \in [0,\infty).\]
Let $\Delta_n$
be the unique solution to
$e^{\Delta_n}
\Delta_n = \kappa_n-n$
and 
define
\[A =c
\left(\frac{2}{\alpha}\right)^{1-1/\alpha}
\Gamma\left(
\frac{1}{\alpha} \right)
\left(\frac{1}{\lambda^{1/\alpha}}
\right).\]
Then, as point processes on $\mathbb R$,
\[
\left\{
2(\kappa_n - n)
\left(|x_i^{(n)}| - 1 - 
\frac{\Delta_n}{2(\kappa_n-n)} \right): 1 \leq i \leq n
\right\}
\xrightarrow[n \to \infty]{\mathrm{law}}
\mathcal P,
\]
where $\mathcal P$ is a Poisson point process
on $\mathbb R$
with intensity $(1+A)^{-1}e^{-s}
\mathrm d s$.
Furthermore,
\[
\lim_{n \to \infty}
\mathbb P\left(2(\kappa_n -n)
\left(M_n - 1 - 
\frac{\Delta_n}{2(\kappa_n - n)} \right)
\leq t
\right)
=
\exp\left(-(1+A)^{-1} e^{- t }\right)  \]
for every $t \in \mathbb R$.
\end{theorem}

Notice that if $\alpha=1$
we must have $c=0$.
As in the case of Theorem \ref{th:Strongly}
we could have chosen $\Delta_n =
\log(\kappa_n - n) - \log \log(\kappa_n - n)$
but we preferred to emphasize 
the property $e^{\Delta_n}
\Delta_n \sim \kappa_n-n$ by using
$\Delta_n$ that satisfies exactly $e^{\Delta_n}
\Delta_n = \kappa_n-n$.

To make the conditions on Theorem \ref{th:Strongly}
and on Theorem \ref{th:WeaklyKappa}
more intuitive we compare them
to the following setting. Let $\nu$ be a 
radial locally finite positive measure on $\mathbb C$ 
and define
\begin{equation}
\label{eq:LogarithmicPotential}
V^{\nu}(r)=\int_1^r\frac{\nu(D_s)}{s}\mathrm d s,
\end{equation}
where $D_s$ denotes the open disk centered
at zero and of radius $s$. 
Notice that $\Delta 
\left(V^\nu(|\cdot|) \right)= 2\pi \nu$
which tells us that $V^\nu$ has the right of
being called an \textit{electrostatic potential}
of $-\nu$.
Suppose that $\nu(\overline{\mathbb D})=1$
and $\nu(D_r)<1$ for every $r<1$. Then
\eqref{eq:LogarithmicPotential} implies that
$V^{\nu}$ satisfies the standard properties
from Definition \ref{def:StandardProperties}. 
Furthermore, 
up to an additive constant,
we have the well-known relation in physics
$V^{\nu}(|x|) = 
\int_{\overline{\mathbb D}} \log|x-y|\mathrm d \nu(y)$
for every $x \in \overline{\mathbb D}$.
This can be obtained either by
noticing that both functions 
have the same Laplacians or
by Fubini's theorem.
See, 
for instance, \cite[Lemma 4.1]{ButezGZ}.
Assume $V^{\nu}(0) > -\infty$ 
for $V^\nu$ to be continuous.
Then,
Frostman's conditions
implies that 
the equilibrium measure
$\mu_{V^\nu}$ is $\nu|_{\overline{\mathbb D}}$.
Using \eqref{eq:LogarithmicPotential} we get that
\[\lim_{r \to 1^-}
\frac{1-\nu(D_r)}{(1-r)^{\alpha-1}} = \lambda >0
\quad \mbox{ implies } \quad
V^\nu(r) = \log r + \frac{\lambda}{\alpha} (1-r)^\alpha
+o(1-r)^\alpha
\mbox{ as } r\to 1^{-}\]
so that the regularity of $\mu_{V^\nu}$ at the unit
circle would tell us
some behavior of  $V^\nu(r)$ near $r=1$.
If, in addition, $\nu(\mathbb C)=1$ then
\eqref{eq:LogarithmicPotential}
implies that $V(r) = \log r$ for $r\geq 1$.
These are precisely the conditions of
Theorem \ref{th:WeaklyKappa}.
Moreover, by using 
\eqref{eq:LogarithmicPotential} we see that
\[
\nu(D_r) = 1 + \lambda (r-1)^{\alpha-1}
+ o (r-1)^{\alpha-1+\varepsilon}
\quad \mbox{ implies } \quad
V^\nu(r) = \log r + \frac{\lambda}{\alpha} (r-1)^\alpha
+o(r-1)^{\alpha+\varepsilon}\]
so that the conditions of
Theorem \ref{th:Strongly} appear
from the regularity of $\nu$ at the unit circle.
A famous example is the quadratic potential
that appears also in the Ginibre ensemble.
We may obtain it by choosing $\nu = 
\ell_{\mathbb C}/\pi$
so that $V^\nu(r) = 
\int_1^r\frac{\nu(D_s)}{s}\mathrm d s =
(r^2-1)/2$ and  $\alpha = \lambda = 2$.

\subsection{Exponential behavior}

\label{sub:Exponential}

Here the system of particles is restricted
to live in the unit disk. The fluctuations
obtained will be exponential and the methods used
are motivated by
\cite{Seo}. 
A peculiarity of the following case
is that the parameter of the exponential
not only depends
on what happens at the edge but it also
depends on what 
happens inside the unit disk. Moreover,
the conditions
on $V$ imply that $\mu_V(\partial \mathbb D) > 0$
so that
the scale $n^2$ needed to obtain
the fluctuations of the maxima
is no longer guessable from $\mu_V$. This is in contrast
to the case where the potential satisfies
\eqref{eq:LeftPotential} for $\alpha > 1$. In
this case,
the scale $n^{2/\alpha}$ may be guessed
from the behavior of $\mu_V$ near the unit circle.
We have stated the $\alpha>1$ version
as an appendix
in Subsection \ref{subAp:Hard-edgeAlphaCase}
 for convenience of the reader.

In the following theorem, $V$ will be a continuous
function on $[0,1]$.

\begin{theorem}[Hard-edge potential]
\label{th:HardEdge}
Suppose $V$
satisfies the \emph{standard properties} and
\[
\lim_{r \to 1^-}\frac{V(r)}{r-1}
= 1-q,\]
for some $q > 0$.
Suppose that $\kappa_n/n \longrightarrow 1$.
Define
\[Q = \max \{p \geq 0: V(s) \geq (1 - p)\log s
\mbox{ for all } s \leq 1\}\]
 and $A = q^2 - (q-Q)^2$.
Then, as point processes on $[0,\infty)$,
\[
\left\{
n^2
(1 - |x_i^{(n)}|) : 1 \leq i \leq n
\right\}
\xrightarrow[n \to \infty]{\mathrm{law}}
\mathcal P,
\]
where $\mathcal P$ is a Poisson point process
on $[0,\infty)$ with intensity $A$.
In particular, for every $t \geq 0$,
\[\mathbb P\left(n^2 \left(1-
M_n
\right) \leq t\right) \to 
1 - 
e^{-At}.\]

\end{theorem}


Following the ideas used in \cite{Seo}, the proof
will involve a limit kernel calculation at the
edge, which is is stated in the following proposition.
A version for $\alpha>1$ can be obtained
by the same methods which can be seen from
the proof of
 Theorem \ref{th:HardEdgeAlpha} in the appendix.
By defining
\[a_k^{(n)}
= \left(2\pi\int_0^1
s^{2k + 1} e^{-2\kappa_n V(s)} 
\mathrm d s\right)^{-1}\]
for $k \in \{0,\dots,n-1\}$, we have the following.

\begin{proposition}[Hard-edge kernel]
\label{prop:HardEdgeKernelLimit}
Under the conditions of Theorem \ref{th:HardEdge},
\[\lim_{n \to \infty} \frac{1}{n^2}
\sum_{k=0}^{n-1} a_k^{(n)}\left(
1 + \frac{z}{n}\right)^k \left(1 + 
\frac{\bar w}{n}\right)^k =e^{z + \bar w}\int_0^Q
\frac{q-t}{\pi} e^{-t(z+\bar w)} \mathrm d t\]
for $(z,w)$ uniformly on compact subsets of 
$\mathbb C \times \mathbb C$.

\end{proposition}

%
%

Notice that 
$0<Q\leq q$.
There is a nice, physical,
interpretation for
$q$ and $Q$.
If there exists some
signed measure $\mu$ that 
has $V$ as its potential, i.e.,
such that
\[V(r) = \int_1^r \frac{\mu(D_s)}{s}
\mathrm d s\]
then we can see that
\[1-q=
\mu(\mathbb D)
\quad \mbox{ while } \quad
1-Q=
\mu_V(\mathbb D).\]
The latter equality
is a consequence
of Frostman's conditions.

\begin{remark}[Distance to the unit circle]
By the same method used to prove 
Theorem \ref{th:HardEdge},
the limiting kernel of
 Proposition \ref{prop:StronglyKernelLimit}
or Proposition \ref{prop:WeaklyKernelLimit}
allows us to find
a Poissonian local limit for the process
of moduli at $r=1$ in cases that are not
hard-edge systems. Nevertheless, since
the point processes obtained
in those cases are Poissonian on $\mathbb R$
with constant
intensity, they give us
no information on the maxima of the moduli.
\end{remark}

\section{Determinantal structure}

\label{sec:determinantal}

The main important structure
is the one of a determinantal point process that
we recall in this section. We begin in
Subsection \ref{sub:DetCoulomb}
by explaining the case
of two-dimensional radial determinantal Coulomb gases.
Then, in Subsection
\ref{sub:GeneralDPP}, we recall the definition of
point process and recall
the definition and some properties of 
general determinantal
determinantal point process.
Finally, in Subsection \ref{sub:Kostlan},
we state Kostlan's idea about the description
of the process of moduli which will
be essential in the proofs.

\subsection{Determinantal radial Coulomb gases}
\label{sub:DetCoulomb}
Let $n$ be a positive integer
and let $\sigma$ be a finite radial
positive measure on 
$\mathbb C$ such that
\[\int_{\mathbb C} |z|^{2k}\mathrm d \sigma(z)
< \infty
\quad \mbox{ for } k \leq n-1.\]
Then, 
$\mathcal Z =
\int_{\mathbb C^n} \prod_{i<j}^n |x_i - x_j|^2 
\mathrm d \sigma^{\otimes_n}(x_1,\dots,x_n)
< \infty$.
Moreover, if we define
\[K(z,w) = \sum_{k=0}^{n-1}
a_k z^k \bar w^k
\quad \mbox{ where } \quad
a_k 
= \left( \int_{\mathbb C} |z|^{2k}\mathrm d \sigma(z)\right)^{-1},
\]
we have that
\[\frac{1}{n!}
\det \left(K(x_i,x_j)_{i,j \in \{1,\dots,n\}}
\right)=\frac{1}{\mathcal Z}
\prod_{i<j}^n |x_i - x_j|^2 .
\]
Furthermore, since
\[\int_{\mathbb C}K(x,y) K(y,z)
\mathrm d \sigma(y) = K(x,z) \quad
\mbox{ and } \quad \int_{\mathbb C}K(x,x)
\mathrm d \sigma(x) = n,\]
we can see that
\[\int_{\mathbb C}
\det \left(K(x_i,x_j)_{i,j \in \{1,\dots,k\}}
\right) \mathrm d \sigma(x_k)
=(n-k+1)
\det \left(K(x_i,x_j)_{i,j \in \{1,\dots,k-1\}}
\right)\]
which implies the following property. If 
$(X_1,\dots,X_n) \sim \frac{1}{\mathcal Z}
\prod_{i<j}^n |x_i - x_j|^2 \mathrm d\sigma^{\otimes_n}(x_1,\dots,x_n)$
then, for every bounded measurable
function $f: \mathbb C \to
\mathbb R$,
\[
\mathbb E\left[\prod_{i=1}^n\left(1+f(X_i)\right)
\right]
=\sum_{k=0}^\infty\frac{1}{k!}
\int_{\mathbb C^k}
\det \left(K(x_i,x_j)_{i,j \in \{1,\dots,k\}}
\right)
\prod_{i=1}^k
f(x_i)
\mathrm d \sigma^{\otimes_k}(x_1,\dots,x_k),
\]
where the sum is actually finite since
$\det \left(K(x_i,x_j)_{i,j \in \{1,\dots,k\}}\right)
= 0$ for $k > n$. In this case,
we say that $\{X_1,\dots,X_n\}$ is
a determinantal point process
with kernel $K$ with respect to 
$\sigma$.

\subsection{General determinantal point processes}
\label{sub:GeneralDPP}

Let us recall what we mean
by a point process or a random configuration of points.
Let $M$ be a locally compact Polish space
and let $\mathcal C_M$ be the space
of positive
measures that take integer values on compact sets. 
Every 
$\mathcal X \in \mathcal C_M$ can be written
as
\[\mathcal X = 
\sum_{\lambda \in \Lambda} \delta_{x_\lambda}\]
for some locally finite family
$\{x_{\lambda}\}_{\lambda \in \Lambda}$
of elements of $M$. We shall usually think
$\mathcal X$ as a multiset and write
$\mathcal X=\{x_{\lambda}:\lambda \in \Lambda\}$. 
For any 
compactly supported measurable
function $f:M \to \mathbb R $ we may define
$\hat f: \mathcal C_M \to \mathbb R$ by
\[\hat f(\mathcal X) = \int_M f \mathrm d \mathcal X
= \sum_{\lambda \in \Lambda} f(x_\lambda).\]
Then, $\mathcal C_M$ will be endowed with
the smallest topology that makes
$\hat f$ continuous for every
compactly supported continuous function 
$f:M \to \mathbb R$. A \textit{point process}
will be a random element of $\mathcal C_M$
and the weak convergence or convergence
in law 
of point processes will be well-defined.

Now, to define the notion of determinantal point
process, let us endow $M$ with a 
locally finite measure $\sigma$
and let us consider a continuous Hermitian function 
$K:M \times M \to \mathbb C$, where
Hermitian means that $K(x,y) = \overline{K(y,x)}$
for every $x, y \in M$.
A point process
 $\mathcal X$ of $M$ is called
  a \textit{determinantal point process} with kernel 
$K$
with respect to the reference measure
$\sigma$ if, for every compactly supported
continuous function $f: M \to \mathbb R$,
\begin{equation}
\label{eq:GeneralMultFunctional}
\mathbb E\left[\prod_{x \in \mathcal X}
\left(1+f(x)\right)
\right]
=\sum_{k=0}^\infty\frac{1}{k!}
\int_{M^k}
\det \left(K(x_i,x_j)_{i,j \in \{1,\dots,k\}}
\right)
\prod_{i=1}^k
f(x_i)
\mathrm d \sigma^{\otimes_k}(x_1,\dots,x_k).
\end{equation}
In fact,
we may relax the continuity 
and the Hermiticity condition of $K$ but
it will not be really 
necessary for our purposes. We refer to 
\cite{HoughKrishnapurPeresVirag}
and \cite{ShiraiTakahashi} for more information
on determinantal point processes.

There are three aspects we wish to
remark.
\begin{itemize}
\item The following freedom on $K$ and $\sigma$
will be useful to consider. Notice that
if $\mathcal X$ is a determinantal 
point process with kernel $K$
with respect to $\sigma$ and if $H:M \to 
\mathbb C\setminus\{0\}$
is continuous, then $\mathcal X$
is a determinantal 
point process with kernel $\widetilde K$
and reference measure $\widetilde \sigma$
given by
\[\widetilde K(x,y)=
H(x) K(x,y) \overline{H(y)} \quad \mbox{ and }
\quad \mathrm d \widetilde \sigma(x) 
= |H(x)|^{-2} \mathrm d \sigma(x).\]
This will be used in the case of radial determinantal
Coulomb gases
to move the potential term
from $\sigma$ to $K$ so that we may choose
the Lebesgue measure or any other convenient
measure as a reference measure. 

\item The defining property
\eqref{eq:GeneralMultFunctional} 
allows us to perform many calculations.
In particular,
it allows us to write
\[\mathbb E \left[ \sum_{x \in \mathcal X}
f(x)\right] = \int_M K(x,x) \mathrm d \sigma(x).\]
The function $\rho(x)= K(x,x)$ is known
as the intensity function
or the first correlation function. The
determinants
$(x_1,\dots,x_k) \mapsto \det \left(K(x_i,x_j)_{i,j \in \{1,\dots,k\}}
\right)$ give the so-called
$k$-correlation functions. Since they
will not be used here, we refer to
\cite{HoughKrishnapurPeresVirag} for more information.

\item One of the most interesting aspects of
determinantal point process
is the continuity of the map that
takes a kernel $K$ to the law of a
point process. More precisely,
if for each $n$ we are given
a determinantal point process $\mathcal X_n$ 
with kernel $K_n$ with respect to the 
reference measure $\sigma$ and if $\mathcal X$
is a
determinantal point process 
with kernel $K_n$ with respect to the 
same reference measure $\sigma$,
\cite[Proposition 3.10]{ShiraiTakahashi} 
tells us that
\begin{equation}
\label{eqDet:ContinuityKernelProcess}
K_n \xrightarrow[n \to \infty]{\mathrm{compact-open}}
K \quad \Longrightarrow
\quad \mathcal X_n 
\xrightarrow[n \to \infty]{\mathrm{law}}
\mathcal X,
\end{equation}
where $\mathrm{compact-open}$
means that the convergence
is in the compact-open topology or, equivalently,
that it is uniform on 
compact sets of $M \times M$.

\end{itemize}

\subsection{Process of the moduli and Kostlan's idea}
\label{sub:Kostlan}
Now, suppose that $M$ is a subset of $\mathbb C$
and suppose that $M$ 
 and $\sigma$ are
 invariant under rotations.
Given $I \subset \mathbb Z$,
let $K$ be the kernel of
the orthogonal projection from
$L^2(M,\sigma)$ onto
the closure of the subspace generated by
$\{z^k\}_{k \in I}$, where we have
assumed that 
$z^k \in L^2(M,\sigma)$ for every
$k \in I$.
To be more precise,
\[K(z,w) = \sum_{k\in I}
a_k z^k \bar w^k,
\quad \mbox{ where }
a_k 
= \left( \int_{\mathbb C} |z|^{2k}\mathrm d \sigma(z)\right)^{-1}.
\]
Suppose that $\mathcal X$ is a determinantal point process
with kernel $K$
with respect to $\sigma$.
Then, an explicit calculation
such as the one in 
\cite[Theorem 4.7.1]{HoughKrishnapurPeresVirag}
or
\cite[Theorem 1.2]{ChafaiPeche}
shows that, if $f: M \to \mathbb R$
is invariant under rotations,
\[
\mathbb E\left[\prod_{x \in \mathcal X}
\left(1+f(x)\right)
\right]
=\prod_{k \in I}\left(1+
a_k
\int_{\mathbb C} f(z) |z|^{2k} \mathrm d \sigma(z)
\right).
\]
In particular,
if $\{\widetilde Y_k\}_{k \in I}$
are independent random variables on $M$
such that 
$\widetilde Y_k \sim a_k |z|^{2k} 
\mathrm d \sigma(z)$, we have 
the equality in law
\begin{equation}
\label{eq:LawEqualityModulus}
\{|x|: x \in \mathcal X\}
\sim \{|\widetilde Y_k|:
k \in I\},
\end{equation}
where we are thinking both
as point processes in $[0,\infty)$. This is the
main property used
in this article and it was initially considered by
Kostlan in \cite{Kostlan}.
The extension where $K$ is not an orthogonal
projection is
described in
\cite[Theorem 4.7.1]{HoughKrishnapurPeresVirag}
and will be particularly useful in the proof of
Theorem \ref{th:MaxFiniteParticlesRandom} in
the appendix. 

For a more concrete example, we may consider
the case from Subsection \ref{sub:DetCoulomb} 
where $\sigma$ is defined by
$\mathrm d \sigma(z) = 
e^{-V(|z|)}\mathrm d \ell_{\mathbb C}(z)$ for some continuous function
$V: [0,\infty) \to \mathbb R$.
If $(X_1,\dots,X_n)$ has a density
proportional to $\prod_{i<j}^n
|x_i - x_j|^2 e^{-2\sum_{i=1}^n V(|x_i|)}$, then
$\{X_1,\dots,X_n\}$ is a determinantal point
process with kernel
\[K(z,w) = \frac{1}{2\pi}\sum_{k=0}^{n-1}
b_k z^k \bar w^k e^{-V(z)} e^{-V(w)},
\quad \mbox{ where } \quad
b_k 
= \left( \int_{0}^\infty r^{2k+1} e^{-2 V(r)}
\mathrm d r \right)^{-1},\]
with respect to Lebesgue measure.
By taking $n$ positive
random variables $\{Y_k\}_{0 \leq k \leq n-1}$
such that $Y_k \sim b_k r^{2k+1}e^{-2V(r)}\mathrm d r$,
we have that
\[\{|X_1|,\dots,|X_n|\} \sim 
\{Y_k: 0 \leq k \leq n-1\}.\]
We may quickly see, for instance, 
that both point processes have the same
intensity
\[\bar \rho(r) = \sum_{k=0}^{n-1}b_k r^{2k+1}
e^{-2V(r)}\]
with respect to Lebesgue measure on $[0,\infty)$, 
i.e., that $\mathbb E[\sum_{k=0}^{n-1} f(Y_k)]$
and
$\mathbb E[\sum_{i=1}^{n} f(X_i)]$ equal
$\int_0^\infty f(r) \bar \rho (r) \mathrm d r $
for every bounded measurable function
$f:[0,\infty) \to \mathbb R$.

\section{Proofs of the Coulombian behavior}
\label{sec:Minima}

To prove Theorem \ref{th:MaxAnnulus} and
Theorem \ref{th:MaxFiniteParticles}
we will take advantage of the continuity
property stated in
\eqref{eqDet:ContinuityKernelProcess}. It
involves an understanding
of the limiting point processes at the correct
scale and in the correct region of $\mathbb C$.
Nevertheless, using the notation 
of Subsection \ref{sub:GeneralDPP},  since 
the map $\sup: \mathcal C_{[0,\infty)}
\to [0,\infty]$ that takes a configuration
of points to its supremum is not 
continuous,
a point process convergence will not imply
the convergence in law
of the maxima. For this reason, it is more
convenient to consider the inverted model 
and to use
the map 
$\inf: \mathcal C_{[0,\infty)}
\to [0,\infty]$ that takes a configuration
of points to its infimum, which is continuous.
Behind this simple reasoning is
the idea that
the maximum and the minimum
are indistinguishable on the sphere:
one can take one to the other by a rotation.
The rotation used will be precisely
the map $z \mapsto 1/z$ in the complex plane.
This idea was used in \cite{ButezGZ}
and it is motivated
by an equivariant property
of the Coulomb gas model which in turn 
is motivated
by the regular case where
the Laplacian of $V$
is thought of as a $(1,1)$-form and
$e^{-2V}$ is thought of as a metric
on the line bundle of
degree one
on the sphere.
These objects can be found
in the work of Berman \cite{Berman} who consider
analogous processes on complex manifolds.
We emphasize
that no complex geometry is needed
in this article
but that the ideas fit nicely in that context.

Let me begin by stating the following
equivariant
property of Coulomb gases.

\begin{lemma}[Inversion of Coulomb gases]
\label{lem:Inversion}
Let $V: \mathbb C \to (-\infty,\infty]$ 
be a measurable function and let $\chi > 0$.
Define $\widetilde V: \mathbb C \setminus \{0\}
\to (-\infty,\infty]$ by
\[ \widetilde{V}(x) = V \left(\frac{1}{x} \right) + 
\log |x|.
 \]
Then, the image of the measure
\[ \prod_{i<j}^n |x_i - x_j|^2 
e^{-2(n+\chi)\sum_{i=1}^n V(x_i)}
{\rm d}\ell_{ \mathbb C}(x_1)\dots
{\rm d}\ell_{ \mathbb C}(x_n)\]
under the application
$(x_1,\dots,x_n) \mapsto (1/x_1,\dots,1/x_n)$
is the measure
\[ \prod_{i<j}^n |x_i - x_j|^2 
e^{-2(n+\chi)\sum_{i=1}^n \widetilde V(x_i)}
{\rm d}\Lambda_{\chi}(x_1)\dots 
{\rm d}\Lambda_{\chi}(x_n)\]
where
\[ {\rm d}\Lambda_{\chi}(x) = |x|^{2(\chi-1)}
{\rm d}\ell_{\mathbb C}(x) \]
\end{lemma}

\begin{proof}
To avoid possible mistakes, we divide 
the change of variables in two steps.
Consider
the function
$G^V:\mathbb C \setminus \{0\}
\times \mathbb C \setminus \{0\}
\to (-\infty,\infty]$ and
the positive measure $\pi$ defined by
\[G^V(x,y) = -\log|x-y| + V(x) + V(y)
 \ \ \
 \mbox{ and }
 \ \ \
 {\rm d}\pi=e^{-2(\chi+1)V }{\rm d}\ell_{\mathbb C}.\]
Then
we may write
\begin{align*}
 \prod_{i<j}^n &|x_i - x_j|^2 
e^{-2(n+\chi)\sum_{i=1}^n V(x_i)}
{\rm d}\ell_{ \mathbb C}(x_1)\dots
{\rm d}\ell_{\mathbb C}(x_n)						\\
&=
\exp 
\left(-2  \left[ -\sum_{i<j}^n \log |x_i-x_j| + 
(n+\chi) \sum_{i=1}^n V(x_i)  \right]
\right) {\rm d}\ell_{\mathbb C}(x_1) \dots
{\rm d}\ell_{\mathbb C}(x_n)			\\
&= e^{-2\sum_{i<j}^n G^V(x_i,x_j)} 
{\rm d}\pi(x_1)\dots {\rm d}\pi(x_n).
\end{align*}
It is enough, then, to notice that
the image of $G^V$ and $\pi$ under the inversion
are $G^{\widetilde V}$
and $\widetilde \pi$, respectively, defined by 
\[G^{\widetilde V}(x,y)
	=-\log|x-y| + \widetilde V(x) + \widetilde V(y)
	\ \ \ 
	\mbox{ and }
	\ \ \
	{\rm d} \widetilde \pi = e^{-2(\chi + 1) V} 
	{\rm d}\Lambda_{\chi}. \]
\end{proof}

Lemma \ref{lem:Inversion} tells us that
the inverted model we will be interested
in will have the following form.
Let us consider
a sequence
$\{\kappa_n\}_{n \geq 1}$ of positive numbers,
a lower semicontinuous
function $U: [0,\infty) \to \mathbb R$
such that
\[\lim_{r \to \infty} 
\left\{U(r) - \log r\right\} > -\infty\]
and
a system of particles
$(x_1^{(n)},\dots,x_n^{(n)} )$
distributed according to
the law proportional to
\begin{equation}
\label{eq:LawKappaNMenosN}
 \prod_{i<j}^n |x_i - x_j|^2 
e^{-2\kappa_n \sum_{i=1}^n U(|x_i|)}
{\rm d}\Lambda_{\kappa_n - n}(x_1)\dots
{\rm d}\Lambda_{\kappa_n - n}(x_n),
\end{equation}
where $\Lambda_\chi$ is defined by
\[ {\rm d}\Lambda_{\chi}(x) 
= |x|^{2(\chi -1)} 
{\rm d}\ell_{\mathbb C}(x) \]
for every $\chi > 0$.
This can be thought of as an electrostatic system
confined by $U$ plus
a potential 
$\left(1 - \frac{n+1}{\kappa_n} \right)\log|\cdot|$
generated by a point charge at the origin.
The following proposition
will be the limiting point process
convergence used to prove
Theorem \ref{th:MaxAnnulus}.
It is stated in a slightly more general form
since there will be no essential difference
in the proof.

\begin{proposition}[Particles at zero potential]
\label{propC:MinA}

Suppose $U:[0,\infty) \to [0,\infty]$ 
is non-negative 
and lower semicontinuous.
Denote 
\[\mathcal A= 
\{ r \geq 0 : \, U(r) = 0\}
\quad \mbox{ and } \quad
R = \emph{ess\,sup } \mathcal A,\]
i.e., 
$R$ is such that the Lebesgue measure of
$\mathcal A \cap (R, \infty)$  is zero
but the Lebesgue measure of
$\mathcal A \cap (r, \infty)$
is different from zero for every $r < R$.
Assume $R>0$ 
and that, for some $\chi > 0$,
\[\kappa_n = n + \chi + 
o(1).\]
Then, 
\[
 \{x_k^{(n)}: k \in \{1,\dots,n\}
 	\mbox{ and } |x_k^{(n)}| < R\}
\xrightarrow[n \to \infty]
	 {\mathrm {law}} \mathbb M,\]
where $\mathbb M$ is
the (inclusion into the open unit
disk of radius $R$ of a) 
determinantal point process
on $\{ x \in \mathbb C : \,|x| < R
\mbox{ and } U(|x|) = 0\}$ 
with kernel
\[ K
(z,w) = \sum_{k=0}^\infty a_k z^k \bar w^k,
 \ \ \
 a_k 
=\left(2 \pi 
\int_{\mathcal A}
r^{2k+2\chi-1} {\rm d}r \right)^{-1},\]
with respect to the reference measure
$\Lambda_{\chi}$.
\begin{proof}
For simplicity, we shall assume that $R\leq 1$.
The general case can be obtained
by a scaling.
Denote 
	$Z = \{x \in \mathbb C:
	|x| < R \mbox{ and } U(|x|)=0\}
		$.
We will prove that
\begin{equation}
\label{eq:ZeroPotentialLimitInProof}
\{x_k^{(n)}: k \in \{1,\dots,n\}
 	\mbox{ and } x_k^{(n)} \in Z \}
 \xrightarrow[n \to \infty]
 {\mathrm{law}} 
 \mathbb M
 \end{equation}
and that
\begin{equation}
\label{eq:NumberGoesToZero}
\# \{x_k^{(n)}: k \in \{1,\dots,n\}
	\mbox{, } 
 	|x_k^{(n)}| \leq \widetilde R 
 	\mbox{ and } x_k^{(n)} \notin Z \}
  \xrightarrow[n \to \infty]
  {\mathrm{law}} 0
  \end{equation}
for every
$\widetilde R < R$. Then we conclude by
the following lemma.

\begin{lemma}[Union with an empty point process]
Let $X$ be a Polish space
and let $C \subset X$ be a closed subset
of $X$.
Suppose we have a sequence of point 
processes
$\{P_n\}_{n \in \mathbb N}$ on $X$
and a point process $P$ on $C$
such 
\[P_n \cap C 
\xrightarrow[n \to \infty]{\mathrm{law}}
 P \ \ \ \mbox{ and } \ \ \
\# \left( P_n \cap \mathcal K \cap C^c\right)
\xrightarrow[n \to \infty]{\mathrm{law}}
 0 \mbox{ for every compact set } 
 \mathcal K \subset X.\] 
Then,
\[P_n 
\xrightarrow[n \to \infty]{\mathrm{law}} 
P,\]
where $P$ is seen as a point process
in $X$ by the natural inclusion.
\begin{proof}[Proof of the lemma]
By \cite[Theorem 4.11]{Kallenberg} we
have to prove that
	$$\sum_{x \in P_n} f(x) 
	 \xrightarrow[n \to \infty]
	 {\mathrm{law}} 
	\sum_{x \in P} f(x)$$
for every continuous function
$f: X \to \mathbb R$ with compact support. 
We already know that
	$$\sum_{x \in P_n \cap C} f(x) 
	\xrightarrow[n \to \infty]
	 {\mathrm{law}} 
	\sum_{x \in P} f(x)$$
so that it is enough, by Slutsky's theorem,
to prove that
	$$\sum_{x \in P_n \cap C^c} f(x) 
	\xrightarrow[n \to \infty]
	 {\mathrm{law}} 
	0.$$
Let $\mathcal K = \mbox{supp} \, f$. Then,
by hypothesis,
$\# \left( P_n \cap \mathcal K \cap C^c\right)
\to 0$.	We can use that
$$\left|\sum_{x \in P_n \cap C^c} f(x) \right|
\leq \# \left( P_n \cap \mathcal K \cap C^c\right)
	\|f\|_{\infty}$$
to conclude.	

\end{proof}
\end{lemma}

To prove
\eqref{eq:ZeroPotentialLimitInProof}
we will use
\cite[Proposition 3.10]{ShiraiTakahashi}
(recalled in \eqref{eqDet:ContinuityKernelProcess}).
A small problem appears since
$\Lambda_{\kappa_n - n}$ is not a fixed
reference measure. We choose any 
$\varepsilon > 0$ and
fix
instead another reference
measure $\Lambda_{\chi - \varepsilon}$.
The point process
$\{x_k^{(n)}: k \in \{1,\dots,n\}
 	\mbox{ and } x_k^{(n)} \in Z \}$
 	is determinantal with kernel 
 	$K_n:Z \times Z \to \mathbb C$ given by
\begin{equation}
\label{KernelFormulaInZeroPotentialProof}
K_n(z,w) = \sum_{k=0}^{n-1} a_k^{(n)}
z^k \bar w^k 
|z \bar w|^{\kappa_n - n - \chi + \varepsilon},
\quad
a_k^{(n)}  = 
\left(2\pi
\int_0^\infty
r^{2k+2\kappa_n - 2n -1} 
e^{-2\kappa_n U\left(r \right)}
{\rm d}r \right)^{-1},
\end{equation}
with respect to $\Lambda_{\chi - \varepsilon}$.
Since $\kappa_n - n - \chi + \varepsilon$ is
positive for $n$ large enough, $K_n$
is continuous and we are in the context
described in Subsection \ref{sub:GeneralDPP}.
We first notice that 
$|z \bar w|^{\kappa_n - n - \chi + \varepsilon}$
converges to $|z \bar w|^{\varepsilon}$
for $(z,w)$
uniformly on compact sets of 
$\mathbb C \times \mathbb C$.
We will now show that
\[\sum_{k=0}^{n-1} a_k^{(n)}
z^k \bar w^k
\xrightarrow[n \to \infty]{}
\sum_{k=0}^\infty a_k z^k \bar w^k\]
for $(z,w)$ uniformly on compact sets 
of $D_R \times D_R$. 
This will be done by showing the
convergence of the coefficients
and by bounding
each term of the series to apply Lebesgue's dominated
convergence theorem.
%
The bound
$r^{2k+2\kappa_n - 2n -1} 
e^{-2\kappa_n U\left(r \right)}
\leq (r^{\varepsilon} + r^{-\varepsilon})r^{2k+2\chi-1} 
e^{-2(k+1+(\chi-\varepsilon))
U\left(r \right)}$,
for $\varepsilon>0$ small enough, allows us to use Lebesgue's
dominated convergence theorem and obtain
\[\int_0^\infty
r^{2k+2\kappa_n - 2n -1} 
e^{-2\kappa_n U\left(r \right)}{\rm d}r
\xrightarrow[n \to \infty]{}
\int_{\mathcal A}
r^{2k+2\chi-1} {\rm d}r.\]
This gives us
the convergence of $a_k^{(n)}$ towards
$a_k$. Now, we need to bound the terms in 
the series. 
By definition of $\mathcal A$ and since
we are assuming that $R=\mathrm{ess\,sup }\,
\mathcal A\leq 1$,
we may choose $\varepsilon > 0$ small enough such that
\[\int_0^\infty
r^{2k+2\kappa_n - 2n -1} 
e^{-2\kappa_n U\left(r \right)}{\rm d}r
\geq
\int_{\mathcal A}
r^{2k+2\kappa_n - 2n -1}  {\rm d}r
\geq
\int_{\mathcal A}
r^{2k+2\chi + \varepsilon -1}   \mathrm d r
\]
for $n$ large enough 
and for every $k \in \{0,\dots,n-1\}$. So, 
we define
\[A_k = \left(2\pi
\int_{\mathcal A} 
r^{2k+2\chi + \varepsilon-1} {\rm d}r \right)^{-1}\]
which satisfies
$a_k^{(n)} \leq A_k$ and notice that,
by Laplace's method, 
\[\frac{1}{k}\log 
\int_{\mathcal A} 
r^{2k+2\chi + \varepsilon -1}  {\rm d}r
=
\frac{1}{k}\log 
\int_{\mathcal A} 
e^{k\log r^2}r^{2\chi + \varepsilon -1}{\rm d}r 
\xrightarrow[k \to \infty]
	 {} \sup \{\log r^2\}\]
where the supremum is taken
over the support of the Lebesgue measure on 
$\mathcal A$.
By the definition of $R$ this supremum
is $\log R^2$ and
the radius of convergence of
$\sum_{k=0}^\infty A_k r^k$ is
$R^2$.
Take $r \in [0,R)$ and suppose
that $|z|, |w| \leq r$.	
Then
$$\left|\sum_{k=0}^{n-1} a_k^{(n)}
z^k \bar w^k
-
\sum_{k=0}^\infty a_k z^k \bar w^k
\right|
\leq
\sum_{k=0}^{\infty} |a_k^{(n)} - a_k|
|z|^k |\bar w|^k
\leq
\sum_{k=0}^{\infty} |a_k^{(n)} - a_k|
r^{2k}
 $$	
where we have defined $a_k^{(n)} = 0$
for $k \geq n$.
Since $|a_k^{(n)} - a_k|r^{2k}$
is bounded by
$2A_k r^{2k}$ 
we can use Lebesgue's dominated convergence
theorem to conclude.

Now, we need to show \eqref{eq:NumberGoesToZero}.
As explained in Section \ref{sec:determinantal},
$\{x_1^{(n)}, \dots, x_n^{(n)} \}$
is a determinantal point process with
respect 
to
$e^{-2\kappa_n U(|z|)}\mathrm d
\Lambda_{\chi - \varepsilon}$ and
with kernel $K_n:\mathbb C \times \mathbb C
\to \mathbb C$ defined by the same formula
\eqref{KernelFormulaInZeroPotentialProof}.
In particular, we have
\[
\mathbb E \left[\# \{x_k^{(n)}: k \in \{1,\dots,n\}
	\mbox{, } 
 	|x_k^{(n)}| \leq \widetilde R 
 	\mbox{ and } x_k^{(n)} \notin Z \}
 	\right]
 	= \int_{Z^c \cap \overline D_{\widetilde R}}
 	 K_n(z,z) e^{-2\kappa_n U(|z|)}
 	\mathrm d \Lambda_{\chi-\varepsilon}(z),
\]
where $\overline D_{\widetilde R}$ 
denotes the closed unit disk of radius $\widetilde R$
and centered at zero.
Since $K_n(z,z)$ converges uniformly on
$\overline D_{\widetilde R}$ and since
$U(|z|)>0$ for $z \in Z^c$, we can use
Lebesgue's dominated convergence
theorem to conclude that
$$\mathbb E \left[\# \{x_k^{(n)}: k \in \{1,\dots,n\}
	\mbox{, } 
 	|x_k^{(n)}| \leq \widetilde R 
 	\mbox{ and } x_k^{(n)} \notin Z \}
 	\right]
\xrightarrow[n \to \infty]
	 {} 0$$
and then
\[\#\{x_k^{(n)}: k \in \{1,\dots,n\}
	\mbox{, } 
 	|x_k^{(n)}| \leq \widetilde R 
 	\mbox{ and } x_k^{(n)} \notin Z \}
 	\xrightarrow[n \to \infty]
	 {\mathrm{law}} 0.\]
	 
\end{proof}
\end{proposition}

Recall that $(x_1^{(n)},\dots,x_n^{(n)} )$
is
distributed according to
the law proportional to
\eqref{eq:LawKappaNMenosN}.
For Theorem \ref{th:MaxFiniteParticles},
and also for Theorem
\ref{th:MaxFiniteParticlesRandom} in the appendix, the following  will be used.

\begin{proposition}[Finite limiting process at zero]
\label{propC:MinFiniteOne}
Suppose that $U:[0,\infty) \to 
[0,\infty)$ is continuous
and that 
\begin{itemize}
\item
	\itemEqNoLabel{U(r) > 0
 		\mbox{ for } r \neq 1 ,	}
\item
	\itemEqNoLabel{
	\lim_{r \to 0}\frac{1}{r^\alpha}U(r)  
	= \lambda,}
\item
	\itemEqNoLabel{
	\lim_{r \to 1^+}\frac{U(r)}{r-1}
= \lp \mbox{ and} 
	}
\item
	\itemEqNoLabel{
	\lim_{r \to 1^-}\frac{U(r)}{1-r}
= \lm
	}
\end{itemize}
for some positive numbers
$\lambda, \alpha, \lp, \lm > 0$. Suppose that,
for some $\chi \in (0, \alpha/2]$
and $\xi \in \mathbb R$,
\[\kappa_n = n + \chi + \frac{\xi}{\log n}
+ o \left(\frac{1}{\log n} \right) .\]
Then,
\begin{equation}
\label{eq:MinFiniteOne}
 \{n^{1/\alpha} x_k^{(n)}: k \in \{1,\dots,n\}
	\}
\xrightarrow[n \to \infty]
	 {\mathrm{law}} 
	 \mathbb G,
\end{equation}
where $\mathbb G$ is
a determinantal point process
on $\mathbb C$ 
with respect to
$\Lambda_{\chi}$ and with kernel
\[
 K
(z,w) = \sum_{k=0}^\infty a_k z^k \bar w^k
e^{-\lambda|z|^\alpha} e^{-\lambda|w|^\alpha}, \]
where
$$\left( a_k \right)^{-1}
= \left\{ \begin{array}{ll} \infty 
		& \mbox{ if } 2k+2\chi > \alpha			\\
2\pi\int_0^\infty r^{2k+2\chi-1}
 e^{-2\lambda r^\alpha}{\rm d}r
	 	& \mbox{ if } 2k+2\chi< \alpha			\\
\pi\left(\frac{1}{\alpha \lambda}+ 
e^{2\xi/\alpha}\left(\frac{1}{\lpp}
+\frac{1}{\lmm} \right)\right)
		& \mbox{ if } 2k+2\chi = \alpha		
\end{array} 
\right. 
.$$
The number of particles 
$\# \mathbb G$
of $\mathbb G$ 
belongs to the interval 
$[\alpha/2 - \chi, \alpha/2-\chi+1]$.
More precisely, 
\[\# \mathbb G
=\lceil \alpha/2 - \chi \rceil
\quad \mbox{ if }
\quad \alpha/2 - \chi \notin 
\mathbb Z,\]
while
$\# \mathbb G \in 
\{\alpha/2-\chi,\, \alpha/2-\chi + 1\}$
with
\[\mathbb P\left(\# \mathbb G 
= \alpha/2 - \chi +1 \right)
=\left(1+\alpha\lambda
e^{2\xi/\alpha}
\left(\frac{1}{\lp} + \frac{1}{\lm}
\right) \right)^{-1}
\mbox{ if }
\quad \alpha/2 - \chi \in 
\mathbb Z.
\]
\end{proposition}

\begin{remark}[Finite Coulomb gas]
Notice that $\mathbb G$, conditioned
to having a fixed number of particles, 
can be thought of as a finite
Coulomb gas confined
by a power-law potential plus a potential multiple of
$\log|\cdot|$.
\end{remark}

\begin{proof}
Notice that 
$\{n^{1/\alpha} x_k^{(n)}: k \in \{1,\dots,n\}\}$
is a determinantal point process 
on $\mathbb C$ associated
to the measure $\Lambda_{\kappa_n - n}$ and to the
kernel
	\[K_n(z,w) = \sum_{k=0}^{n-1} a_k^{(n)}
z^k \bar w^k 
e^{-\kappa_n U\left(\frac{|z|}{n^{1/\alpha}} \right)}
e^{-\kappa_n U\left(\frac{|w|}{n^{1/\alpha}} \right)},
\]
where
\[
\left(a_k^{(n)} \right)^{-1} = 
\int_{\mathbb C}|z|^{2k}
e^{-2\kappa_n
U\left( \frac{|z|}{n^{1/\alpha}} \right)}
{\rm d}\Lambda_{\kappa_n-n}(z)
=
2\pi
\int_0^\infty
r^{2k+2\kappa_n - 2n -1} 
e^{-2\kappa_n U\left(\frac{r}{n^{1/\alpha}} \right)}
{\rm d}r.\]
Since the reference measure is different
for each $n$ we will consider, instead, the fixed
reference measure $\Lambda_{\chi-\varepsilon}$,
for some (choose any) 
$\varepsilon \in (0, \chi)$.
The kernel will be
\[K_n(z,w) = \sum_{k=0}^{n-1} a_k^{(n)}
z^k \bar w^k 
e^{-\kappa_n U\left(\frac{|z|}{n^{1/\alpha}} \right)}
e^{-\kappa_n U\left(\frac{|w|}{n^{1/\alpha}} \right)}
|z \bar w|^{\kappa_n - n - \chi + \varepsilon},
\]
By \cite[Proposition 3.10]{ShiraiTakahashi}, 
to obtain \eqref{eq:MinFiniteOne}
we should
understand the limit of
the sequence of continuous (for $n$ large enough)
functions
$\{K_n\}_{n \geq 1}$.
Since $\lim_{r \to 0}\frac{1}{r^\alpha}U(r)  
= \lambda \in (0, \infty)$ we already have 
that,
uniformly on compact sets,
	\[\kappa_n U \left( \frac{|z|}{n^{1/\alpha}} 
	\right) \to \lambda|z|^{\alpha}.\]
Since $\kappa_n - n \to \chi$ we know that
$|z|^{\kappa_n - n - \chi + \varepsilon}
\to |z|^{\varepsilon}$
uniformly on compact sets.
Then
what is left to prove is that,
uniformly on compact sets,	
\[\sum_{k=0}^{n-1} a_k^{(n)}
z^k \bar w^k
\to
\sum_{k=0}^\infty a_k z^k \bar w^k.\]
This will be done by showing the convergence
of the coefficients, by 
properly bounding them and then
by applying Lebesgue's dominated convergence theorem.
We want to find the limit,
as $n$ goes to infinity, of
\begin{equation}
\label{eq:akIntegral}
\left( 2\pi a_k^{(n)} \right)^{-1} = \int_0^\infty
r^{2k+2\kappa_n - 2n-1} 
e^{-2 \kappa_n U\left(\frac{r}{n^{1/\alpha}} \right)}
{\rm d}r
=
n^{(2k + 2\chi)/\alpha}\int_0^\infty
r^{2k + 2\kappa_n - 2n -1} 
e^{-2 \kappa_n U\left(r \right)}
{\rm d}r
\end{equation}
By Laplace's method
the dominant part will be around
the minima of $U$ (i.e., $r = 0$ and $r=1$).
This suggests us to split the right-hand side
integral in \eqref{eq:akIntegral} as
\[\int_0^\infty
= \int_0^{1/2}
+ \int_{1/2}^1
+ \int_1^2
+ \int_2^\infty.\]
We begin by understanding the integral from 
$0$ to $1/2$. The main contribution comes
from $r = 0$ so that we perform
a change of variables and apply
Lebesgue's dominated convergence theorem 
(by controlling
the potential $U(r)$ from below by a multiple of 
$r^\alpha$),
\begin{align*}
n^{(2k + 2\kappa_n - 2n)/\alpha} 
\int_0^{1/2}
r^{2k+2\kappa_n - 2n -1} 
e^{-2\kappa_n U\left(r \right)}
{\rm d}r					
&=
\int_0^{\frac{n^{1/\alpha}}{2}}
r^{2k+2\kappa_n - 2n -1} 
e^{-2\kappa_n U\left(\frac{r}{n^{1/\alpha}} \right)}
{\rm d}r			
			\\
&\xrightarrow[n \to \infty]{}
\int_0^\infty
r^{2k+2\chi-1} 
e^{-2\lambda r^\alpha }
{\rm d}r.
\end{align*}
For the integral from $1/2$ to $1$ we notice
that, by applying
Lebesgue's dominated convergence theorem
(by controlling the potential from below by
a multiple of $1 - r$),
\begin{align*}
	n
	\int_{1/2}^{1} r^{2k+2\kappa_n - 2n -1} 
	e^{-2\kappa_n U\left( r\right)} 
	{\rm d}r
	&=
	\int_{0}^{n/2}
	\left(1-\frac{r}{n} \right)^{2k+2\kappa_n - 2n-1} 
	e^{-2\kappa_n U\left(1-\frac{r}{n}\right)} {\rm d}r
	\\	
	&\xrightarrow[n \to \infty]{}
	\int_{0}^{\infty}
	e^{-2 \lmm r} {\rm d}r	
	=\frac{1}{2 \lm}	
\end{align*}
which implies that
	\[n^{(2k+2\kappa_n -2n)/\alpha}
	\int_{1/2}^{1} r^{2k+2\kappa_n
	-2n-1} 
	e^{-2 \kappa_n U\left( r\right)} {\rm d}r
	\xrightarrow[n \to \infty]{}
	\left\{ \begin{array}{ll}
	\infty & \mbox{ if } 2k + 2\chi  > \alpha	\\
	e^{2\xi/\alpha}\frac{1}{2 \lmm} 
			& \mbox{ if } 2k + 2\chi = \alpha	\\
	0 & \mbox{ if } 2k + 2\chi< \alpha	
	\end{array} \right.
	.\]
Similarly, for the integral from $1$ to $2$ we have
	\[n^{(2k+2\kappa_n - 2n)/\alpha}
	\int_1^{2} r^{2k+2\kappa_n - 2n -1} 
	e^{-2\kappa_n U\left( r\right)} {\rm d}r
	\xrightarrow[n \to \infty]{}
	\left\{ \begin{array}{ll}
	\infty & \mbox{ if } 2k + 2\chi > \alpha	\\
	e^{2\xi/\alpha}\frac{1}{2 \lpp} 
			& \mbox{ if } 2k + 2\chi = \alpha	\\
	0 & \mbox{ if } 2k + 2\chi  < \alpha	
	\end{array} \right. 
	.\]
The last integral, the one from $2$ to $\infty$,
will go to zero since there is $\varepsilon > 0$
such that
$\varepsilon \log r \leq U(r)$
for $r \geq 2$ which implies that
\begin{align*}
n^{(2k + 2\kappa_n - 2n)/\alpha}\int_{2}^\infty
& r^{2k+2\kappa_n - 2n -1} 
e^{-2 \kappa_n U\left(r \right)}{\rm d}r			\\
&\leq
n^{(2k + 2\kappa_n - 2n)/\alpha}\int_{2}^\infty
r^{2k+2\kappa_n - 2n-1} 
e^{-2\kappa_n \varepsilon \log r}{\rm d}r 
	\xrightarrow[n \to \infty]{} 0.
\end{align*}
In summary, we have obtained
that
\[\lim_{n \to \infty} (a_k^{(n)})^{-1}
= \left\{ \begin{array}{ll}
	\infty 
		& 
		\mbox{ if }	2k+2\chi > 	\alpha	
								\\
	2\pi\int_0^\infty r^{2k+2\chi-1} 
	e^{-2\lambda r^\alpha}{\rm d}r
 		& 
 		\mbox{ if }	2k+2\chi < \alpha	
 								\\
	2\pi\int_0^\infty r^{2k+2\chi-1} 
	 e^{-2\lambda r^\alpha}{\rm d}r
	+
	\pi e^{2\xi/\alpha}
	 \left( \frac{1}{ \lpp}
	+\frac{1}{ \lmm} \right)
		 & 
		 \mbox{ if } 2k+2\chi = \alpha		
\end{array} 
\right.\]
To control the coefficients we notice that
by the behavior of $U$ near the origin
and since $U$ is bounded in $[0,1]$, there exists $C>0$ and $\bar \chi > 0$ such that
if we define
\[A_k = 
	\left(2\pi\int_0^{k^{1/\alpha} } 
	r^{2k+ 2\bar \chi - 1} 
	e^{-C r^\alpha} {\rm d}r \right)^{-1}\]
we have
\[A_k \leq a_k^{(n)} \mbox{ for } k \in 
\{0,\dots,n-1\}.\]
By Laplace's method,
$\lim_{k \to \infty}
\frac{1}{k}\log \left[ (A_k)^{-1} \right]
=\infty$
so that
$\sum_{k=0}^\infty A_k x^k$  has an infinite
radius of convergence. Finally,
if $r > 0$ and $|z|, |w| \leq r$ we have
$$\left|\sum_{k=0}^{n-1} a_k^{(n)}
z^k \bar w^k
-
\sum_{k=0}^\infty a_k z^k \bar w^k
\right|
\leq
\sum_{k=0}^{\infty} |a_k^{(n)} - a_k|
|z|^k |\bar w|^k
\leq
\sum_{k=0}^{\infty} |a_k^{(n)} - a_k|
r^{2k}
 $$	
where $a_k^{(n)} $ is zero
if $k \geq n$.
By noticing that $|a_k^{(n)} - a_k|r^{2k}
\leq 
2A_k r^{2k}$ 
we apply Lebesgue's dominated convergence
theorem to conclude.

\vspace{1.5mm}
\noindent
\textbf{Number of particles.} 
The
assertion about the
number of particles 
is an immediate consequence
of \cite[Theorem 4.5.3]{HoughKrishnapurPeresVirag}.
More precisely, when
$\alpha/2-\chi \notin \mathbb Z$,
the operator defined by $ K$
is a projection 
onto a space of dimension
$\lceil \alpha/2-\chi \rceil$,
while if 
$\alpha/2-\chi \in \mathbb Z$,
the unique
non-zero eigenvalue less than one
is the inverse of
$1+\alpha\lambda
e^{2\xi/\alpha}
(\frac{1}{\lp} + \frac{1}{\lm}
)$.

\end{proof}

Having all the ingredients,
we proceed to the proof of Theorem \ref{th:MaxAnnulus}
and Theorem \ref{th:MaxFiniteParticles}.

\begin{proof}[Proof of Theorem \ref{th:MaxAnnulus}]
We shall use Proposition \ref{propC:MinA} with
 $U:[0,\infty) \to \mathbb R$ defined by
\[U(r) = 
\left\{ \begin{array}{ll}
V(1/r) + \log r & \mbox { if } r > 0 \\
0 & \mbox { if } r = 0
\end{array}
\right. .
\]
Notice that $U$ is a 
non-negative lower semicontinuous function
and that 
$\mathcal A = [\widetilde R^{-1}, R^{-1}]$.
By Lemma \ref{lem:Inversion},
$(1/{x_1^{(n)}},\dots,1/{x_n^{(n)}})$
has a law proportional to
\[
 \prod_{i<j}^n |x_i - x_j|^2 
e^{-2\kappa_n \sum_{i=1}^n U(|x_i|)}
{\rm d}\Lambda_{\kappa_n - n}(x_1)\dots
{\rm d}\Lambda_{\kappa_n - n}(x_n),
\]
where
\[ {\rm d}\Lambda_{\kappa_n - n}(x) 
= |x|^{2(\kappa_n - n -1)} 
{\rm d}\ell_{\mathbb C}(x) .\]
Proposition \ref{propC:MinA} tells us that
\[\{1/x_k^{(n)}: k \in \{1,\dots,n\}
\mbox{ and }
|1/x_k^{(n)}|< R^{-1}\} \xrightarrow[n \to \infty]
{\mathrm{law}} \mathbb M\]
where $\mathbb M$ is
the (inclusion into the open unit
disk of radius $R^{-1}$ of a) 
determinantal point process
on $\{ x \in \mathbb C : \, \widetilde R^{-1}
\leq |x| < R^{-1}\}$ 
with kernel
\[ K
(z,w) = \sum_{k=0}^\infty a_k z^k \bar w^k,
 \ \ \
\left( a_k \right)^{-1}
=2 \pi 
\int_{\widetilde R^{-1}}^{R^{-1}}
r^{2k+2\chi-1} {\rm d}r,\]
with respect to the
reference measure
$\Lambda_{\chi}$.
In particular,
\[\{|1/x_k^{(n)}|: k \in \{1,\dots,n\}
\mbox{ and }
|1/x_k^{(n)}|< R^{-1}\} \xrightarrow[n \to \infty]
{\mathrm{law}} 
\{|x|: x\in \mathbb M\}.\]
Suppose that $\{T_k\}_{k \geq 0}$
is a family
of independent random variables
in $[\tilde R^{-1},R^{-1})$ such that 
the density of $T_k$
is given by
$r \in [\tilde R^{-1},R^{-1}) \mapsto 
 2\pi a_k r^{2k+2\chi -1} $.
Then, by
 \eqref{eq:LawEqualityModulus},  we obtain that
$\{|x|: x\in \mathbb M\} \sim 
\{T_k: k \geq 0\}$.
By defining $Y_k = 1/T_k$ and taking
the image by $x \mapsto 1/x$ we get
\[\{|x_k^{(n)}|: k \in \{1,\dots,n\}
\mbox{ and }
|x_k^{(n)}| > R\} \xrightarrow[n \to \infty]
{\mathrm{law}} 
\{Y_k: k \geq 0\},\]
which is what we wanted to prove.
For the assertion about the maxima, we begin
with
\[\min\{|1/x_k^{(n)}|: k \in \{1,\dots,n\}
\mbox{ and }
|1/x_k^{(n)}|< R^{-1}\} \xrightarrow[n \to \infty]
{\mathrm{law}} 
\inf \{T_k: k \geq 0\},\]
where the minimum is $R^{-1}$
if the set is empty. 
Since
\[\min\{|1/x_k^{(n)}|: k \in \{1,\dots,n\}
\mbox{ and }
|1/x_k^{(n)}|< R^{-1}\} 
= \min\{|1/x_k^{(n)}|: k \in \{1,\dots,n\}\}
\wedge R^{-1} \]
and since $\inf \{T_k: k \geq 0\}
< R^{-1}$ almost surely,
we obtain that
\[\inf\{|1/x_k^{(n)}|: k \in \{1,\dots,n\}
\} \xrightarrow[n \to \infty]
{\mathrm{law}} 
\inf \{T_k: k \geq 0\},\]
and the proof may be completed by taking the image
by $x\mapsto 1/x$.
\end{proof}

\begin{proof}[Proof of Theorem \ref{th:MaxFiniteParticles}]
It follows the same steps
as the proof of Theorem
\ref{th:MaxAnnulus} above by using
Proposition \ref{propC:MinFiniteOne}
instead of
Proposition \ref{propC:MinA}.
\end{proof}

\section{Proofs of the Gumbel behavior}
\label{sec:ProofsGumbel}

In this section we will give the proofs
of the statements from
Subsection \ref{sub:Gumbel}. Namely,
we give
the proofs of
Theorem \ref{th:Strongly}, Theorem
\ref{th:WeaklyKappa}, Proposition
\ref{prop:StronglyKernelLimit} and
Proposition \ref{prop:WeaklyKernelLimit}.

We begin by describing, in Subsection
\ref{sub:IdeaOfProofGumbel},
the proofs of
Theorem \ref{th:Strongly} and Theorem
\ref{th:WeaklyKappa}. 
In Subsection \ref{sub:ControlProbabilities},
we describe some lemmas that
allow us to properly justify the steps. 
The convergence of the intensities,
which is the main step of the proofs, is provided
in Subsection \ref{sub:IntensitiesConvergence}.
Finally, in Subsection \ref{sub:ProofsOfGumbelBehavior},
we complete the proofs of
the Gumbel fluctuations. 

\subsection{Idea of the proofs of 
Theorem \ref{th:Strongly} and Theorem
\ref{th:WeaklyKappa}}

\label{sub:IdeaOfProofGumbel}

For each $n \geq 1$, let us consider
$n$ independent positive random variables
$Y_0^{(n)},\dots,Y_{n-1}^{(n)}$
such that 
\[Y_k^{(n)} \sim b_k^{(n)} 
r^{2k+1}e^{-2\kappa_n V(r)} 1_{(0,\infty)}(r)
{\rm d} r,
\quad 
\mbox{ where } \quad
b_k^{(n)} 
=\left(\int_{0}^\infty
s^{2k+1}e^{-2\kappa_n V(s)}
\mathrm d s
\right)^{-1}.\]
By
Kostlan's idea stated in 
Subsection \ref{sub:Kostlan}, we have that
\begin{equation}
\label{eq:KostlanStatementInGumbel}
\{|x_1^{(n)}|,\dots,|x_n^{(n)}|\}
\sim
\{Y_0^{(n)},\dots,Y_{n-1}^{(n)}
\}.
\end{equation}
Let us begin by
studying 
$M_n = \max \left\{|x_1^{(n)}|,\dots,|x_n^{(n)}|
\right\}$.  By
\eqref{eq:KostlanStatementInGumbel},
we have
\[\mathbb P
\left(M_n \leq m\right)
= \prod_{k=0}^{n-1}
\left(1-b_k^{(n)} \int_{m}^\infty
r^{2k+1} e^{-2\kappa_n V(r)}
\mathrm d r \right).\]
By using that
$\log(1-x) = -x + O(x^2)$,
we can see that if $m_n$ goes to $1$ slowly enough 
such that
\begin{equation}
\label{eq:UniformConvCondition}
\sup_{k \in \{0,\dots,n-1\}}
\left(b_k^{(n)}\int_{m_n}^\infty
r^{2k+1} e^{-2\kappa_n V(r)}
\mathrm d r \right)
\xrightarrow[n \to \infty]{} 0,
\end{equation}
we have that
\begin{equation}
\label{eq:LogConvergesIfIntensityDoes}
\log \mathbb P(M_n \leq m_n)
\mbox{ converges if and only if } 
- \bigintss_{m_n}^\infty
\sum_{k=0}^{n-1}
b_k^{(n)} r^{2k+1} e^{-2\kappa_n 
V(r)}
\mathrm d r \mbox{ converges}
\end{equation}
and their limit is the same.
We recognize the first intensity of
the process of moduli
\[\bar \rho_n(r) = 
\sum_{k=0}^{n-1}
b_k^{(n)} r^{2k+1} e^{-2\kappa_n V(r)}\]
as minus the integrand. 
If
 $m_n = c_n + d_n x$
 for some 
$\{c_n\}_{n \geq 1}$ and
$\{d_n\}_{n \geq 1}$,
we would write
\[\int_{m_n}^\infty
\bar  \rho_n(r)
\mathrm d r
=
\int_{x}^\infty 
d_n \, 
\bar \rho_n\left (c_n + d_n s \right)
\mathrm d s.\]
So, we only need to find
\[\lim_{n\to\infty} d_n \, 
\bar \rho_n\left (c_n + d_n s \right)
= \rho(s)\]
and control the
integrand to obtain that
\[\lim_{n \to \infty}
\mathbb P
\left(M_n \leq c_n + d_n x\right)
=
\exp \left(-\int_x^\infty
\rho(s) \mathrm d s \right).\]
Notice that $s \mapsto
d_n \, 
\bar \rho_n\left (c_n + d_n s \right)$ 
is 
the intensity function
of the point process
\[\big\{
d_n^{-1}(|x_1^{(n)}|
-c_n) , \dots,
d_n^{-1} (|x_n^{(n)}|
-c_n)
\big\}\]
and 
that
$\exp \left(-\int_x^\infty
\rho(s) \mathrm d s \right)$
is the
cumulative distribution function
of the maximum
of a Poisson point process
with intensity $\rho$.
This suggests how to proceed
for the study of the point process
formed by $W_k^{(n)} = d_n^{-1}
(Y_k^{(n)} - c_n)$. It is motivated
by \cite{Rider2,Seo}.
Let us consider $\ell$ mutually disjoint
bounded
measurable subsets
$I_1,\dots,I_\ell$ of $\mathbb R$ and
define the random variables
\[N_i^{(n)} = \sum_{k=0}^{n-1} 
1_{I_i}(W_k^{(n)}).\]
We shall study
$\mathbb P \big(N_1^{(n)}=m_1,
\dots,N_\ell^{(n)}=m_\ell\big)$.
For this, by using that
\[1_{N_i^{(n)} = m_i}
= 
\left. \frac{1}{m_i!}
\left(\frac{\mathrm d}{\mathrm d \lambda_i}
\right)^{m_i}
\left[
\prod_{k=0}^{n-1}\left(1+\lambda_i
1_{I_i}(W_k^{(n)}) \right)\right]
\right|_{\lambda_i = -1}\]
we can notice that
\begin{align*}
1_{N_1^{(n)} = m_1  ,\dots,N_\ell^{(n)} = m_\ell}&	
\\
&\hspace{-20mm}= 
\left. \frac{1}{m_1!}
\dots \frac{1}{m_\ell!}
\left(\frac{\mathrm d}{\mathrm d \lambda_1}
\right)^{m_1} 
\dots
\left(\frac{\mathrm d}{\mathrm d \lambda_\ell}
\right)^{m_\ell}
\left[
\prod_{k=0}^{n-1}\left(1+
\sum_{i=1}^\ell \lambda_i
1_{I_i}(W_k^{(n)})
\right)\right]
\right|_{\lambda_1=\dots=\lambda_\ell = -1}.
\end{align*}
Then, we can write
\begin{align*}
\mathbb P
(N_1^{(n)} = m_1  ,\dots,N_\ell^{(n)} = m_\ell)
&	
\\
&\hspace{-40mm}= 
\left. \frac{1}{m_1!}
\dots \frac{1}{m_\ell!}
\left(\frac{\mathrm d}{\mathrm d \lambda_1}
\right)^{m_1}
\dots
\left(\frac{\mathrm d}{\mathrm d \lambda_\ell}
\right)^{m_\ell}
\left[
\prod_{k=0}^{n-1}\left(1+
\sum_{i=1}^\ell \lambda_i
\mathbb P(W_k^{(n)}\in I_i) 
\right)\right]
\right|_{\lambda_1=\dots=\lambda_\ell = -1}.
\end{align*}
For each $n$, define 
the polynomial function
$F_n:\mathbb C^\ell \to \mathbb C^\ell$ by
\[F_n(\lambda_1,\dots,\lambda_\ell)=
\prod_{k=0}^{n-1}\left(1+
\sum_{i=1}^\ell \lambda_i
\mathbb P(W_k^{(n)}\in I_i) 
\right).\]
Notice that
$|F_n(\lambda_1,\dots,\lambda_\ell)|
\leq F_n(r,\dots,r)$
whenever $|\lambda_1|,\dots,|\lambda_\ell|
\leq r$. By Montel's theorem, 
once we show that $F_n$ 
converges pointwise, it also converges
uniformly on compact sets of $\mathbb C^\ell$.
Since we are assuming
\eqref{eq:UniformConvCondition}, we have that
\[
\sum_{k=0}^{n-1}\log \left(1+
\sum_{i=1}^\ell \lambda_i
\mathbb P(W_k^{(n)}\in I_i) 
\right)
 \mbox{ converges if and only if } 
\sum_{k=0}^{n-1}\sum_{i=1}^\ell \lambda_i
\mathbb P(W_k^{(n)}\in I_i) \mbox{ converges}
\]
and their limit is the same. Here the logarithm
is well-defined
for $n$ large enough by the usual power series.
Then, we need to study the limit of
\[
\sum_{k=0}^{n-1}
\mathbb P(W_k^{(n)}\in I_i) 
=\int_{I_i} 
d_n \, 
\bar \rho_n\left (c_n + d_n s \right)
\mathrm d s\]
which, since 
$d_n \, 
\bar \rho_n\left (c_n + d_n s \right)
\to \rho(s)$
and since 
$d_n \, 
\bar \rho_n\left (c_n + d_n s \right)$
has been properly
bounded to study the maxima,
could be shown to converge to
$\int_{I_i} \rho(s) \mathrm d s$.
We get
\[F_n(\lambda_1,\dots,\lambda_\ell)
\xrightarrow[n \to \infty]{}
e^{\sum_{i=1}^\ell
\lambda_i \int_{I_i}\rho(s) \mathrm d s}.\]
Taking the corresponding derivatives
we obtain that
\[\mathbb P
(N_1^{(n)} = m_1  ,\dots,N_\ell^{(n)} = m_\ell)
\xrightarrow[n \to \infty]{}
\prod_{i=1}^\ell 
\left(\frac{\left(\int_{I_i}\rho(s) \mathrm d s\right)^{m_i}}{m_i!}
e^{-\int_{I_i}\rho(s) \mathrm d s}\right)
\]
which implies that
\[\{W_k^{(n)}: k \in \{0,\dots,n-1\}\}
\xrightarrow[n \to \infty]{\mathrm{law}}
\mathcal P,\]
where $\mathcal P$ is a Poisson point process
with intensity $\rho$.

Now, we describe the lemmas 
that allow us to justify the steps
described above.

\subsection{Control of the probabilities}

\label{sub:ControlProbabilities}

Next 
lemma, together with Lemma \ref{lemG:KernelOutside},
imply
\eqref{eq:UniformConvCondition}
which 
justifies \eqref{eq:LogConvergesIfIntensityDoes}
for Theorem \ref{th:Strongly}.

\begin{lemma}[Uniform convergence of probabilities
for a power-law behavior]
\label{lemG:AlphaPotentialUniformBound}
Let $\alpha \geq 1$. 
Suppose that there exists
$\lambda, \bar \lambda, \ell > 0$ such that
\[\bar \lambda 
|r-1|^{\alpha} \leq V(r) - \log r
\leq \lambda |r-1|^{\alpha}\]
when $|r-1|\leq \ell$
and suppose that
 \[ \lim_{n \to \infty}
 \frac{\kappa_n}{n} = 1
 \quad \mbox{ and } \quad
 n-\kappa_n = 
	o \left( \kappa_n^{1/\alpha}
	\right).\]
Define
\[b_k^{(n)} 
=
\left(\int_0^\infty
s^{2k + 1} e^{-2\kappa_n V(s)} 
\mathrm d s\right)^{-1}.
\]
If $\{m_n\}_{n \geq 1}$ is 
a sequence that satisfies
\[n^{1/\alpha} (m_n - 1) 
\xrightarrow[n \to \infty]{}
\infty,\]
then
\[
\sup_{k \in \{0,\dots,n-1\}}
\left(
b_k^{(n)} \int_{m_n}
^{1+\ell}
r^{2k+1} e^{-2\kappa_n V(r)}
\mathrm d r
\right)
\xrightarrow[n \to \infty]{} 0.\]

\end{lemma}

\begin{proof}[Proof of Lemma 
\ref{lemG:AlphaPotentialUniformBound}]

To bound the coefficients $b_k^{(n)}$, we use that
\begin{align*}
\int_0^\infty
s^{2k + 1} e^{-2\kappa_n V(s)} 
\mathrm d s
&\geq 
\int_{1-\ell}^{1+\ell}
s^{2k -2\kappa_n + 1 } 
e^{-2\kappa_n \lambda |s-1|^{\alpha}} 
\mathrm d s					\\
&=
\int_{-\ell}^{\ell}
(1+s)^{2k -2\kappa_n + 1 } 
e^{-2\kappa_n \lambda |s|^\alpha} 
\mathrm d s					\\
&=
\int_{0}^{\ell}
\left((1+s)^{2k -2\kappa_n + 1 } 
+
(1-s)^{2k -2\kappa_n + 1 } \right)
e^{-2\kappa_n \lambda s^\alpha} 
\mathrm d s					\\
&\geq \int_{0}^{\ell}
e^{-2\kappa_n \lambda s^\alpha} 
\mathrm d s					\\
&=\frac{1}{\kappa_n^{1/\alpha}}
 \int_{0}^{\kappa_n^{1/\alpha} \ell }
e^{-2\lambda s^\alpha} 
\mathrm d s.
\end{align*}
For the integral part we can see
that
\begin{align*}
\int_{m_n}
^{1+\ell}
r^{2k+1} e^{-2\kappa_n V(r)}
\mathrm d r
&\leq 
\int_{m_n}
^{1+\ell}
r^{2k - 2\kappa_n +1} e^{-2
\kappa_n
\bar \lambda |r-1|^\alpha}
\mathrm d r					\\
&=
\frac{1}{\kappa_n^{1/\alpha}}
\int_{ \kappa_n^{1/\alpha} (m_n-1)}
^{\kappa_n^{1/\alpha} \ell}
\left( 1 +
\frac{r}{\kappa_n^{1/\alpha}}
\right)^{2k - 2\kappa_n +1} 
e^{-2
\bar \lambda r^\alpha}
\mathrm d r					\\
&\leq
\frac{1}{\kappa_n^{1/\alpha}}
\int_{\kappa_n^{1/\alpha} (m_n-1)}
^{ \kappa_n^{1/\alpha} \ell}
\left( 1 +
\frac{r}{\kappa_n^{1/\alpha}}
\right)^{2n - 2\kappa_n -1} 
e^{-2
\bar \lambda r^\alpha}
\mathrm d r					\\		
&\leq
\frac{1}{\kappa_n^{1/\alpha}}
\int_{\kappa_n^{1/\alpha} (m_n-1)}
^{ \kappa_n^{1/\alpha} \ell}
\exp \left(
\left|\frac{2n - 2\kappa_n -1}{
\kappa_n^{1/\alpha}}\right| r \right)
e^{-2
\bar \lambda r^\alpha}
\mathrm d r					\\
&= o\left(\frac{1}{\kappa_n^{1/\alpha} }\right)	
\quad \mbox{ uniformly on } k
\end{align*}
since $(2n-2\kappa_n - 1)/
\kappa_n^{1/\alpha} \to 0$ and
$\kappa_n^{1/\alpha} (m_n - 1)
\to \infty$. This completes the proof.

\end{proof}

Next lemma is
the step in \eqref{eq:UniformConvCondition}
which
justifies \eqref{eq:LogConvergesIfIntensityDoes}
for Theorem \ref{th:WeaklyKappa}.

\begin{lemma}[Uniform convergence of probabilities
for a `logarithm' potential ]
\label{lemG:LogPotentialUniformBound}
Suppose that
\[V(r) = \log r \quad \mbox{ for } r \geq 1.\]
and that
 \[ \lim_{n \to \infty}
 \frac{\kappa_n}{n} = 1
 \quad \mbox{ and } \quad
 \lim_{n \to \infty}(\kappa_n - n) =
\infty.\]
Define
\[b_k^{(n)} 
=
\left(\int_0^\infty
s^{2k + 1} e^{-2\kappa_n V(s)} 
\mathrm d s\right)^{-1}.
\]
If $\{m_n\}_{n \geq 1}$ is 
a sequence that satisfies
\[(\kappa_n - n) (m_n - 1) 
\xrightarrow[n \to \infty]{}
\infty,\]
then
\[
\sup_{k \in \{0,\dots,n-1\}}
\left(
b_k^{(n)} \int_{m_n}
^{\infty}
r^{2k+1} e^{-2\kappa_n V(r)}
\mathrm d r
\right)
\xrightarrow[n \to \infty]{} 0.\]

\end{lemma}

\begin{proof}

We use that
\[ b_k^{(n)} \leq 
\left(\int_1^\infty
s^{2k + 1} e^{-2\kappa_n V(s)} 
\mathrm d s\right)^{-1} = 2(\kappa_n - k -1).\]
Then, we may to notice that
\begin{align*}
b_k^{(n)} \int_{m_n}
^{\infty}
r^{2k+1} e^{-2\kappa_n V(r)}
\mathrm d r
&\leq
2(\kappa_n - k-1)
\int_{m_n}^{\infty}
r^{2k+1} e^{-2\kappa_n V(r)}
\mathrm d r									
											\\
&\leq
m_n^{2(n - \kappa_n)}				
\end{align*}
which goes to zero as $n \to \infty$
since
$(\kappa_n - n) (m_n - 1) 
\to
\infty$.
\end{proof}

We state now a lemma that helps
us neglect the intensity
outside a neighborhood of $1$.
It will be useful in the proof of
Theorem \ref{th:Strongly}.

\begin{lemma}[Intensity far from $1$]
\label{lemG:KernelOutside}
Let $V: [0,\infty) \to \mathbb R$
be a continuous function such that
\[
\quad V(1) = 0,
\quad \forall r > 1, \, V(r) > \log r, 
\quad \mbox{ and }
\quad \liminf_{r \to \infty}
\frac{V(r)}{\log r} > 1,\]
and suppose that
\[\lim_{n \to \infty}
\frac{\kappa_n}{n}=1.\]
Define
\[b_k^{(n)} 
=
\left(\int_0^\infty
s^{2k + 1} e^{-2\kappa_n V(s)} 
\mathrm d s\right)^{-1}.
\]
Then, for every $\varepsilon > 0$,
\[\sum_{k=0}^{n-1}
b_k^{(n)} 
\int_{1+\varepsilon}^\infty
r^{2k+1} e^{-2\kappa_n V(r)}
\mathrm d r
\to 0.\]

\end{lemma}
\begin{proof}

Choose $C>1$ such that
\[\forall r \geq 1+\varepsilon:
\, V(r) \geq C\log r\]
and choose $\delta>0$ such that
\[\forall r \in [1,1+\delta]:
\, V(r)\leq \log(1+\varepsilon)
\frac{C-1}{2}.\]
Since,
\[\int_0^\infty
s^{2k + 1} e^{-2\kappa_n V(s)} 
\mathrm d s 
\geq
\int_{1}^{1+\delta}
s^{2k + 1} e^{-\kappa_n 
\log(1+\varepsilon) (C-1)} 
\mathrm d s						
\geq
\delta 
(1+\varepsilon)^{-\kappa_n (C-1)}\]
we have
\begin{align*}
b_k^{(n)} 
r^{2k+1} e^{-2\kappa_n V(r)}
&\leq
\delta^{-1}(1+\varepsilon)^{\kappa_n (C-1)}
r^{2k + 1} r^{- 2 \kappa_n C}		\\
&\leq 
\delta^{-1}r^{\kappa_n(C-1)}
r^{2k + 1} r^{-2\kappa_n C}			\\
&=
\delta^{-1}r^{2k + 1} r^{-\kappa_n (C+1)}		\\
&\leq 
\delta^{-1}
r^{2n - 1} r^{-\kappa_n (C+1)}.
\end{align*}
By integrating,
taking the sum and taking the limit,
 the proof
is completed.
\end{proof}

Next lemma complements the previous
one by allowing us to neglect the 
intensity outside a shrinking neighborhood of 
$1$. It works when $V$
behaves like a power of $\alpha$
near both sides of $1$
and it will be used in the proof of
Theorem \ref{th:Strongly}

\begin{lemma}[Intensity near 1]
\label{lemG:AlphaKernelNear}
Let $\alpha \geq 1$. 
Suppose that there exists
$\lambda, \bar \lambda, \ell > 0$ such that
\[\bar \lambda 
|r-1|^{\alpha} \leq V(r) - \log r
\leq \lambda |r-1|^{\alpha}\]
when $|r-1|\leq \ell$
and suppose that
 \[ \lim_{n \to \infty}
 \frac{\kappa_n}{n} = 1
 \quad \mbox{ and } \quad
 n-\kappa_n = 
	o \left( \kappa_n^{1/\alpha}
	\right).\]
Define
\[b_k^{(n)} 
=
\left(\int_0^\infty
s^{2k + 1} e^{-2\kappa_n V(s)} 
\mathrm d s\right)^{-1}.
\]
Then, for every $\gamma > 0$,
\[\sum_{k=0}^{n-1}
b_k^{(n)} 
\int_{1+\frac{1}{(\log n)^\gamma}}^{1+\ell}
r^{2k+1} e^{-2\kappa_n V(r)}
\mathrm d r
\to 0.\]

\end{lemma}

\begin{proof}
By the proof
of Lemma
\ref{lemG:AlphaPotentialUniformBound} we
know that
\[ b_k^{(n)} \leq \widetilde C \kappa_n^{1/\alpha}\]
for some constant $\widetilde C>0$.
Then, we need to study
\begin{align*}\sum_{k=0}^{n-1}
\int_{1+\frac{1}{(\log n)^\gamma}}^{1+\ell}
r^{2k+1} e^{-2\kappa_n V(r)}
\mathrm d r										
&=
\int_{1+\frac{1}{(\log n)^\gamma}}^{1+\ell}
\frac{r^{2n} - 1}{r^2-1}r e^{-2\kappa_n V(r)}
\mathrm d r										\\
&\leq (\log n)^\gamma
\int_{1+\frac{1}{(\log n)^\gamma}}^{1+\ell}
r^{2n + 1} e^{-2\kappa_n V(r)}
\mathrm d r										\\
&\leq (\log n)^\gamma
\int_{1+\frac{1}{(\log n)^\gamma}}^{1+\ell}
r^{2n + 1 - 2 \kappa_n} e^{-2\kappa_n \bar \lambda
(r - 1)^\alpha}
\mathrm d r										\\
&\leq \frac{(\log n)^\gamma}{\kappa_n^{1/\alpha}}
\int_{\frac{\kappa_n^{1/\alpha}}{(\log n)^\gamma}}
^{\kappa_n^{1/\alpha}\ell}
\left(1+\frac{r}{\kappa_n^{1/\alpha}}\right)^{2n + 1 - 2 \kappa_n} e^{-2\bar \lambda
r^\alpha}
\mathrm d r										\\
&\leq \frac{(\log n)^\gamma}{\kappa_n^{1/\alpha}}
\int_{\frac{\kappa_n^{1/\alpha}}{(\log n)^\gamma}}
^{\kappa_n^{1/\alpha}\ell}
e^{r |2n + 1 - 2 \kappa_n|
\kappa_n^{-1/\alpha}} e^{-2\bar \lambda
r^\alpha}
\mathrm d r			.
\end{align*}
Since $(n-\kappa_n)/\kappa_n^{1/\alpha} \to 0$
we can conclude that there is 
$C > 0$ such that
\[\sum_{k=0}^{n-1}
\int_{1+\frac{1}{(\log n)^\gamma}}^{1+\ell}
r^{2k+1} e^{-2\kappa_n V(r)}
\mathrm d r	
\leq \frac{(\log n)^\gamma}{\kappa_n^{1/\alpha}}
\int_{\frac{\kappa_n^{1/\alpha}}{(\log n)^\gamma}}
^{\infty}
 e^{-C
r^\alpha}
\mathrm d r	,\]
for $n$ large enough. Since
\[\int_{\frac{\kappa_n^{1/\alpha}}{(\log n)^\gamma}}
^{\infty}
 e^{- C 
r^\alpha}
\mathrm d r
= e^{-C \frac{\kappa_n}{(\log n)^{\gamma \alpha}}
(1+ o(1))},\]
 we can complete the proof by using that
$(\log n)^\gamma e^{-C \frac{\kappa_n}{(\log n)^{\gamma \alpha}}
(1+ o(1))} \to 0$.
 
\end{proof}

The following lemma
helps us neglect the intensity 
outside a shrinking neighborhood of $1$
for $V$ that equals the logarithm
at the right of $1$.
It will be used in
the proof of Theorem \ref{th:WeaklyKappa}.

\begin{lemma}[Intensity near $1$ for
a `logarithm' potential]
\label{lemG:LogKernelNear}
Suppose that
\[V(r) = \log r \quad \mbox{ for }
r \geq 1\]
and that
 \[ \lim_{n \to \infty}
 \frac{\kappa_n}{n} = 1
 \quad \mbox{ and } \quad
 \lim_{n \to \infty}
 \left(\kappa_n - n \right)=  \infty.\]
Define
\[b_k^{(n)} 
=
\left(\int_0^\infty
s^{2k + 1} e^{-2\kappa_n V(s)} 
\mathrm d s\right)^{-1}.
\]
Then, for every $\gamma >0$,
\[\sum_{k=0}^{n-1}
b_k^{(n)} 
\int_{1+ \frac{\gamma}{\sqrt{\kappa_n - n}}}^{\infty}
r^{2k+1} e^{-2\kappa_n V(r)}
\mathrm d r
\to 0.\]

\end{lemma}

\begin{proof}
We use that
\[ b_k^{(n)} \leq 
\left(\int_1^\infty
s^{2k + 1} e^{-2\kappa_n V(s)} 
\mathrm d s\right)^{-1} = 2(\kappa_n - k -1).\]
Then, we need to study
\begin{align*}
\sum_{k=0}^{n-1} 2(\kappa_n - k-1)
\int_{1+\frac{\gamma}{
\sqrt{\kappa_n - n}}}^{\infty}
r^{2k+1} e^{-2\kappa_n V(r)}
\mathrm d r			
&=
\sum_{k=0}^{n-1}
\left(1+ \frac{\gamma}{\sqrt{\kappa_n - n}}
\right)^{2(k - \kappa_n + 1)}					
											\\
&\leq
\sqrt{\kappa_n - n}
\left(1+ \frac{\gamma}{\sqrt{\kappa_n - n}}
\right)^{2(n - \kappa_n + 1)}			
\end{align*}
which goes to zero as $n \to \infty$.
\end{proof}

\subsection{Convergence of the intensities}

\label{sub:IntensitiesConvergence}

\begin{proposition}[Intensity convergence
for a power-law behavior]
\label{prop:IntensityStrongly}
Under the conditions and notation of
Theorem \ref{th:Strongly}, if we define
\[b_k^{(n)} = \left(\int_0^\infty 
x^{2k+1} e^{-2\kappa_n V(x)} \mathrm d x
\right)^{-1},\]
we have, for every $s \in \mathbb R$,
\[
\frac{\delta_n}{\Delta_n}
\sum_{k=0}^{n-1}
b_k^{(n)} \left(1 + 
\delta_n + \frac{\delta_n}{\Delta_n} s
\right)^{2k+1}
e^{-2 \kappa_n V\left(1 + 
\delta_n + \frac{\delta_n}{\Delta_n} s\right)}
\xrightarrow[n \to \infty]{}
\frac{e^{-2\xi} e^{-s}}{A}
\]
\end{proposition}

\begin{proof}
We begin by noticing that,
since $e^{\Delta_n/\alpha} \Delta_n = 
n^{1/\alpha}$, we have
$\Delta_n \to \infty$. Moreover,
by using that
$\frac{\Delta_n}{\alpha}
\sim \frac{\Delta_n}{\alpha} + \log \Delta_n = 
\frac{1}{\alpha}\log n$
we obtain 
\[\Delta_n \sim \log n.\]
Define
\[c_n =1 + \left(
\frac{\Delta_n}{
2\lambda_{\scaleto{+}{4.5pt}}
\kappa_n}
 \right)^{1/\alpha}
 \quad 
 \mbox{ and }
 \quad d_n = \frac{1}{\Delta_n}
 \left(
\frac{\Delta_n}{
2\lambda_{\scaleto{+}{4.5pt}}
\kappa_n}
 \right)^{1/\alpha}.\]
so that we want to find the limit of
\begin{equation}
\label{eq:RadialIntensityModified}
 d_n 
\sum_{k=0}^{n-1}
b_k^{(n)} (c_n + d_n s)^{2k+1}
e^{-2 \kappa_n V(c_n + d_n s)}.
\end{equation}
For simplicity of notation, we do not write
the subscripts and superscripts $n$
when there is no possibility of confusion.
For the potential term we have
\begin{align*}
V(c + d s)
&= \log (c + d s)
+
\frac{\lambda_{\scaleto{+}{3.5pt}}}{\alpha}
(c + d s - 1)^{\alpha}
+ o(c+d s - 1 )^{\alpha + \varepsilon}		\\
&=
\log (c + d s)
+
\frac{\lambda_{\scaleto{+}{3.5pt}}}{\alpha}
\left[
\frac{\Delta}{2 \lambda_{\scaleto{+}{3.5pt}}
\kappa}
\left(1 + \frac{s}{\Delta} \right)^{\alpha}
\right]
+ o\left(
\frac{\Delta}{
\kappa}
\right)^{1 + \frac{\varepsilon}{\alpha}}			\\
&=
\log (c + d s)
+
\frac{\Delta}{2 \alpha
\kappa}
\left(1 + \frac{\alpha s}{\Delta} 
+ o\left(\frac{1}{\Delta} \right)\right)
+ o\left(
\frac{\Delta}{
\kappa}
\right)^{1 + \frac{\varepsilon}{\alpha}}			\\
&=
\log (c + d s)
+
\frac{\Delta}{2 \alpha
\kappa}
 + \frac{s}{2\kappa} 
+ o\left(\frac{1}{\kappa}\right).
\end{align*}
So, we can write
\[\exp \left(-2\kappa
\left[ 
V\left(c + d s
\right) - \log\left(c + d s\right)
\right]\right)
=\exp \left(-\frac{\Delta}{\alpha} 
- s + o (1)\right)
\sim \frac{\Delta}{n^{1/\alpha}}
e^{- s}.\]
Then, the limit of 
\eqref{eq:RadialIntensityModified}
will be the same as the limit of
\begin{equation*}
e^{- s}
\left(
\left(\frac{\Delta}{
2\lambda_{\scaleto{+}{4.5pt}} \kappa}\right)^{1/\alpha}
\
\sum_{k=0}^{n-1}
\frac{b_k}{n^{1/\alpha}}
 (c + d s)^{2(k - 
 \lfloor \kappa \rfloor)}
 \right).
\end{equation*}
where $ \lfloor \kappa \rfloor$
denotes the integer part of $\kappa$. 
Then, by recalling that
$\delta =  \left(\frac{\Delta}{2\lambda_{\scaleto{+}{3.5pt}}
 \kappa}
  \right)^{1/\alpha}$,
it is enough to prove that $R_n$, defined by
\[ R_n
=\delta
\sum_{k=0}^{n-1}
\frac{b_k}{n^{1/\alpha}}
\left(1 + \frac{s}{\Delta}\delta + \delta
 \right)^{2(k-\lfloor\kappa\rfloor) }
=
 \delta
\sum_{k=-\lfloor\kappa\rfloor}^{
n-1-\lfloor\kappa\rfloor}
\frac{b_{\lfloor\kappa\rfloor + k}}{n^{1/\alpha}}
\left(1 + \frac{s}{\Delta}\delta + \delta
 \right)^{2k},\]
 converges to the constant
 $\frac{e^{-2\xi}}{A}$. If $t^*= 
\delta\lfloor t /\delta\rfloor $,
we may rewrite $R_n$ as
\[R_n= 
\bigints_{-\delta \lfloor \kappa
\rfloor}^{\delta
(n-\lfloor \kappa \rfloor)}
\left(
\frac{b_{\lfloor \kappa \rfloor
+ t^*/\delta}^{}}{n^{1/\alpha}}
\right)
 \left(1 + \frac{s}{\Delta}\delta + \delta
\right)^{2t^{*}/\delta }
\mathrm d t. \]
We look at each term in the integrand. We have

\[
\lim_{n \to \infty}
\left(1+\frac{s}{\Delta}\delta+\delta
\right)^{2t^*/\delta}
=e^{2t}
\]
and, for
\[\left(
\frac{b_{\lfloor \kappa \rfloor
+t^* /\delta}^{}}{n^{1/\alpha}}
\right)^{-1}
=
n^{1/\alpha}\int_0^\infty
x^{2(\lfloor \kappa \rfloor 
+ t^{*} /\delta) + 1} 
e^{-2\kappa V(x)} 
\mathrm d x,	\]
we can split the integral in two regions.
We can 
use
that $V(x) > \log x$ for $x \neq 1$,
 that $V$ is bounded from below
and the strong
confinement condition \eqref{eq:2StronglyConfiningInTheorem},
to see
that
\[\lim_{n \to \infty}
n^{1/\alpha}\int_{[0,1-\varepsilon]
 \cup [1+\varepsilon,\infty)}
x^{2(\lfloor \kappa \rfloor 
+ t^{*} /\delta) + 1} 
e^{-2\kappa V(x)} 
\mathrm d x = 0\]
for any $\varepsilon > 0$.
We have then

\begin{align*}
n^{1/\alpha}\int_{1-\varepsilon}^{1+\varepsilon}
x^{2(\lfloor \kappa \rfloor 
+ t^{*} /\delta) + 1} 
e^{-2\kappa V(x)} 
\mathrm d x							
&=
n^{1/\alpha}\int_{1-\varepsilon}^{1+\varepsilon}
x^{2t^{*} /\delta + 1
+2\left(\lfloor \kappa \rfloor
- \kappa\right)} 
e^{-2\kappa \left(V(x) - \log x \right)} 
\mathrm d x							\\
& \hspace{-17mm}=
\int_{-\varepsilon n^{1/\alpha}}^{\varepsilon 
n^{1/\alpha}}
\left(1 + \frac{x}{n^{1/\alpha}} \right)^{2t^{*}/\delta + 1
+2\left(\lfloor \kappa \rfloor
- \kappa\right)} 
e^{-2\kappa \left(V
\left(1 + \frac{x}{n^{1/\alpha}}\right) 
- \log \left(1 + \frac{x}{n^{1/\alpha}} \right) \right)} 
\mathrm d x							\\
& \hspace{-17mm} \xrightarrow[n \to \infty]{}  
 \int_{0}^\infty
e^{-
\frac{2\lambda_{\scaleto{-}{3.5pt}}}{\alpha} x^\alpha} 
\mathrm d x	+
 \int_{0}^\infty
e^{-\frac{2\lambda_{\scaleto{+}{3.5pt}}}{\alpha} 
x^\alpha} 
\mathrm d x						\\
& \hspace{-17mm} =
\Gamma\left(
\frac{1}{\alpha} \right)
\alpha^{\frac{1}{\alpha}
- 1}
\left(
(2\lambda_{\scaleto{-}{4.5pt}})^{-1/\alpha}
+
(2\lambda_{\scaleto{+}{4.5pt}})^{-1/\alpha}
 \right)=\frac{A}{2}
\end{align*}
where we have bounded
the integrand in
$[-\varepsilon n^{1/\alpha},
\varepsilon n^{1/\alpha} ]$ by bounding
$V(1+x) - \log (1+x) $
from below by a multiple of $|x|^\alpha$
for
$x \in [-\varepsilon,\varepsilon]$ and 
applied Lebesgue's dominated
convergence theorem. 
We may notice that
 there exists $C>0$ and $\varepsilon>0$
such that by defining
\[E_C(x) = \left\{ 
\begin{array}{ll}
\exp(Cx)& \mbox{ if } x \leq 0 \\
\exp(C^{-1}x)& \mbox{ if } x > 0 
\end{array}\right. \]
we have, for $n$ large enough,
\begin{align}\left(
\frac{b_{\lfloor \kappa \rfloor
+ t^*/\delta}^{}}{n^{1/\alpha}}
\right)&
 \left(1 + \frac{s}{\Delta} \delta + \delta
\right)^{2t^{*}/\delta}
1_{\left[-\delta \lfloor \kappa
\rfloor,\delta
(n-\lfloor \kappa \rfloor)\right)}	(t)		\nonumber
	\\
\leq  C &
\left(
\int_{-\varepsilon n^{1/\alpha}}^{0}
E_C\left(\frac{t x}{\delta n^{1/\alpha}}
\right) e^{-C|x|^\alpha} \mathrm d x
\right)^{-1} E_{C^{-1}}
\left(\left(\frac{s}{\Delta} + 1 \right) t\right)
1_{\left[-\delta \lfloor \kappa
\rfloor,\delta
(n-\lfloor \kappa \rfloor)\right)} (t)	\nonumber	
										\\
\leq  C &
\left(
\int_{-\varepsilon n^{1/\alpha}}^{0}
E_C\left(\frac{(n-\lfloor \kappa \rfloor) x}{ n^{1/\alpha}}
\right) e^{-C|x|^\alpha} \mathrm d x
\right)^{-1} E_{C^{-1}}
\left(\left(\frac{s}{\Delta} + 1 \right) t\right)
1_{\left[-\delta \lfloor \kappa
\rfloor,\delta
(n-\lfloor \kappa \rfloor)\right)} (t)	 \nonumber		
										\\
\leq  \widetilde C &
E_{C^{-1}}
\left(\left(\frac{s}{\Delta} + 1 \right) t\right)
1_{\left[-\delta \lfloor \kappa
\rfloor,\delta
(n-\lfloor \kappa \rfloor)\right)} (t)	
\label{eq:IneqIntegrandStrongly},
\end{align}
where we have used
that $(n-\lfloor \kappa \rfloor)/n^{1/\alpha}
\to 0$ to find the bound $\widetilde C$.
Since $s$ is fixed, we may bound the previous
expression to  apply Lebesgue's dominated convergence
theorem and conclude that
\begin{align*}
\bigints_{-\delta \lfloor \kappa
\rfloor}^{\delta
(n-\lfloor \kappa \rfloor)}
\left(
\frac{b_{\lfloor \kappa \rfloor
+ t^*/\delta}^{}}{n^{1/\alpha}}
\right)
 \left(1 + \frac{s}{\Delta} \delta + \delta
\right)^{2t^{*}/\delta}
\mathrm d t
	\xrightarrow[n \to \infty]{}  &
\frac{2}{A}
\int_{-\infty}^{-\xi}
e^{2t}
\mathrm d t			\\
	&=\frac{
e^{-2\xi}}{A}
\end{align*}
 where
 \[\xi = \lim_{n \to \infty}
 \delta(\lfloor \kappa\rfloor - n)
 =
   \lim_{n \to \infty}
\left(
\frac{\Delta}{2
\lambda_{\scaleto{+}{4.5pt}} \kappa}
\right)^{1/\alpha}
 (\kappa-n)
 =
    \lim_{n \to \infty}
\left(\frac{\log n}{2\lambda_{\scaleto{+}{4.5pt}} n} 
\right)^{1/\alpha}
 (\kappa-n).\]
\end{proof}

\begin{proposition}[Intensity convergence
for a `logarithm' potential]
\label{prop:IntensityWeakly}
Under the conditions and notation of
Theorem \ref{th:WeaklyKappa}, 
if we define
\[b_k^{(n)} = \left(\int_0^\infty 
x^{2k+1} e^{-2\kappa_n V(x)} \mathrm d x
\right)^{-1},\]
we have, for every $s \in \mathbb R$,
\[
\frac{1}{2(\kappa_n - n)}
\sum_{k=0}^{n-1}
b_k^{(n)} \left(1 + \frac{\Delta_n}{2(\kappa_n - n)}
 + \frac{s}{2(\kappa_n - n)} 
 \right)^{2k+1}
e^{-2 \kappa_n V
\left(1 + \frac{\Delta_n}{2(\kappa_n - n)}
 + \frac{s}{2(\kappa_n - n)} \right)}
\xrightarrow[n \to \infty]{} 
\frac{e^{-s}}{1+A} .
\]
\end{proposition}

\begin{proof}
First we notice that, since $\kappa_n - n$
goes to infinity and since
$e^{\Delta_n}\Delta_n = \kappa_n - n$,
we obtain that $\Delta_n \to \infty$.
Then, since $\Delta_n \sim \Delta_n + 
\log \Delta_n = \log (\kappa_n - n)$, we get that
\[\Delta_n \sim \log(\kappa_n - n).\]
%
%
Let 
\[c_n = 1 +
\frac{\Delta_n}{2(\kappa_n - n)}
 \quad \mbox{ and }
 \quad
 d_n = \frac{1}{2(\kappa_n - n)}.\]
We want to find the limit
of 
\begin{equation}
\label{eq:2RadialIntensityModifiedWeak}
 d_n 
\sum_{k=0}^{n-1}
b_k^{(n)} (c_n + d_n s)^{2k+1}
e^{-2 \kappa_n V(c_n + d_n s)}.
\end{equation}
For the potential term we have
\begin{align*}
\exp \left(-2\kappa_n
V\left(c_n + d_n s
\right) +2n \log\left(c_n + d_n s\right)\right)	\\
&\hspace{-17mm}=
 \exp \left(2(n-\kappa_n)
 \log
\left(1+\frac{\Delta_n}{2(\kappa_n - n)} +
\frac{s}{2(\kappa_n - n)} \right)  \right)			\\
&\hspace{-17mm}
\sim e^{-\Delta_n} e^{-s} = \frac{\Delta_n}{\kappa_n - n}
e^{-s}.
\end{align*}
If we define
$\delta_n = \frac{\Delta_n}{2(\kappa_n - n)}$,
the limit of 
\eqref{eq:2RadialIntensityModifiedWeak}
will be the same as the limit of
\begin{equation*}
e^{-s}
\left(
\delta_n
\sum_{k=0}^{n-1}
\frac{b_k^{(n)}}{\kappa_n - n}
 \left(1+ \frac{s}{\Delta_n}{\delta_n}
 +\delta_n \right)^{2(k - n) }
 \right).
\end{equation*}
We want to prove that $R_n$, defined by
\[R_n=
\delta_n
\sum_{k=0}^{n-1}
\frac{b_k^{(n)}}{\kappa_n - n}
\left(1 +  \frac{s}{\Delta_n}{\delta_n} + \delta_n
 \right)^{2(k-n)}
=
\delta_n
\sum_{k=-n}^{-1}
\frac{b_{k+n}^{(n)}}{\kappa_n - n}
\left(1 +  \frac{s}{\Delta_n}{\delta_n}+ \delta_n
 \right)^{2k},
\]
converges
to the constant $\frac{1}{1+A}$.
Define $t^* = \delta_n \lfloor t/ \delta_n \rfloor$
so that
\[R_n 
 =
\bigints_{-n\delta_n}^{0}
\frac{b_{t^*/\delta_n + n}^{(n)}}
{\kappa_n - n}
\left(1 +  \frac{s}{\Delta_n}{\delta_n}+ \delta_n \right)^{2t^*/\delta_n}
\mathrm d t .\]
We look at each term in the integrand. We have
\[
\lim_{n \to \infty}
\left(1+\frac{s}{\Delta_n}\delta_n+\delta_n
\right)^{2t^*/\delta_n}
=e^{2t}
\]
and, for
\begin{align*}
\left(\frac{b_{t^*/\delta_n+n}^
{(n)}}{\kappa_n - n}\right)^{-1}
&=(\kappa_n - n)
\int_0^\infty x^{2t^*/\delta_n+2n+1}
e^{-2\kappa_n V(x)}		\mathrm d x
\end{align*}
we can split the integral in three regions.
The integral from $1$ to $\infty$
can be explicitly calculated and we obtain
\[(\kappa_n - n)
\int_1^\infty x^{2t^*/\delta_n+2n+1}
e^{-2\kappa_n V(x)}		\mathrm d x
\xrightarrow[n \to \infty]{}\frac{1}{2}.\]
For the integral from
$0$ to $1-\varepsilon$, for any $\varepsilon>0$,
we may use that
$V(x) > \log x$ for $x<1$ and that
$V$ is bounded from below to obtain
\[(\kappa_n - n)
\int_0^{1-\varepsilon} x^{2t^*/\delta_n+2n+1}
e^{-2\kappa_n V(x)}		\mathrm d x
\xrightarrow[n \to \infty]{} 0.\]
As in
the proof of Proposition \ref{prop:IntensityStrongly},
the integral from $1-\varepsilon$ to $1$
can be understood by 
controlling $V(1+x) - \log(1+x)$
from below
by a multiple of $|x|^\alpha$
and by using Lebesgue's dominated
convergence theorem. 
We obtain
\begin{align*}
(\kappa_n - n)&
\int_{1-\varepsilon}^1 x^{2t^*/\delta_n+2n+1}
e^{-2\kappa_n V(x)}		\mathrm d x
\\
&=
\frac{\kappa_n - n}{n^{1/\alpha}}
\int_{-\varepsilon n^{1/\alpha}}^0
\left(1 + \frac{x}{n^{1/\alpha}}
\right)^{
2\frac{t^*}{\delta_n}+1 +2(n-\kappa_n)}
e^{-2\kappa_n \left[V
\left(1 + \frac{x}{n^{1/\alpha}} \right)
-\log\left(1 + \frac{x}{n^{1/\alpha}} \right)
\right]} \mathrm d x				\\		
&\to
c \int_{-\infty}^0
e^{-\frac{2\lambda}{\alpha} |x|^\alpha} \mathrm d x		\\
&=
c 
\Gamma\left(
\frac{1}{\alpha} \right)
\alpha^{\frac{1}{\alpha}
- 1}
(2\lambda)^{-1/\alpha}=\frac{A}{2}.
\end{align*}
We can see that
 there exists $C>0$ such that, 
for $n$ large enough,
\begin{align}\left(
\frac{b_{
t^*/\delta_n + n}^{}}{\kappa_n - n}
\right)&
 \left(1 + \frac{s}{\Delta_n} \delta_n + \delta_n
\right)^{2t^{*}/\delta_n}
1_{\left[-n \delta_n,0\right)}	(t)		\nonumber
	\\
&\leq   
\frac{\kappa_n- n}{2(\kappa_n - n - 1)} 
\exp
\left(C\left(\frac{s}{\Delta_n} + 1 \right) t\right)
1_{\left[-n \delta_n,0\right)} (t)	\nonumber			
										\\
&\leq 
\exp
\left(C\left(\frac{s}{\Delta_n} + 1 \right) t\right)
1_{\left[-n \delta_n,0\right)} (t)	
\label{eq:IneqIntegrandWeakly}.
\end{align}
By Lebesgue's dominated convergence theorem,
justified by
the bound \eqref{eq:IneqIntegrandWeakly},
we get
\[R_n = 
\bigints_{-n\delta_n}^{0}
\frac{b_{t^*/\delta_n + n}^{(n)}}
{\kappa_n - n}
\left(1 +  \frac{s}{\Delta_n}{\delta_n}
+\delta_n \right)^{2t^*/\delta_n}
\mathrm d t 
\xrightarrow[n \to \infty]{}
\frac{2}{1+A}
\int_{-\infty}^0 e^{2t} 
\mathrm d t
= \frac{1}{1+A},
\]
where we are also using that 
$n\delta_n\to \infty$, consequence
of the fact that
$\frac{\kappa_n - n}{n^{1/\alpha}}$
is bounded and that $\Delta_n \sim 
\log(\kappa_n - n)$.
\end{proof}

\subsection{Proofs of
Theorem \ref{th:Strongly} and
Theorem \ref{th:WeaklyKappa}}

\label{sub:ProofsOfGumbelBehavior}

\begin{proof}[Proof of Theorem
\ref{th:Strongly}]

Recall that
$\delta_n = 
 \left(
\frac{\Delta_n}{
2\lambda_{\scaleto{+}{4.5pt}}
\kappa_n}
 \right)^{1/\alpha} $.
If we define
 $c_n = 1+\delta_n$ and 
$d_n = \delta_n \Delta_n^{-1}$,
 we want to prove that, for every $t \in \mathbb R$,
\[\lim_{n \to \infty}
\mathbb P\left(
M_n
\leq c_n + d_n t
\right)
=
\exp\left(-A^{-1} e^{- t -2 \xi }\right) . \] 
Consider
the intensity
of the process of moduli $\bar \rho_n$ given by
\[\bar \rho_n(r) = 
r\sum_{k=0}^{n-1}
b_k^{(n)} r^{2k} e^{-2\kappa_n V(r)}
\quad \mbox{ where } \quad b_k^{(n)} 
=\left(\int_{0}^\infty
s^{2k+1}e^{-2\kappa_n V(s)}
\mathrm d s
\right)^{-1}\]
By Lemma
\ref{lemG:KernelOutside}
and Lema \ref{lemG:AlphaPotentialUniformBound}
we only need to prove that, for every $x \in \mathbb R$,
\[\lim_{n \to \infty} \int_{x}^\infty 
d_n \, 
\bar \rho_n\left (c_n + d_n s \right)
\mathrm d s =
\exp\left(-A^{-1} e^{- x -2 \xi }\right) \] 
as explained in
Subsection \ref{sub:IdeaOfProofGumbel}. 
By Proposition 
\ref{prop:IntensityStrongly}, we have that
\[d_n \, 
\bar \rho_n\left (c_n + d_n s \right)
\xrightarrow[n \to \infty]{}
\frac{e^{-2\xi} e^{-s}}{A}\]
for every $s \in \mathbb R$. Now,
we need to bound $d_n \, 
\bar \rho_n\left (c_n + d_n s \right)$
and apply Lebesgue's dominated convergence 
theorem. By
Lemma \ref{lemG:KernelOutside} and
Lemma \ref{lemG:AlphaKernelNear} we know that
\[\int_{1+\frac{1}{(\log n)^\gamma}}^\infty
\bar \rho_n(r) \mathrm d r
\xrightarrow[n \to \infty]{} 0\]
so that we only need to control
$d_n \, 
\bar \rho_n\left (c_n + d_n s \right)$
when 
$c_n + d_n s \leq 1+1/(\log n)^\gamma$.
Recall that in the proof of
Proposition \ref{prop:IntensityStrongly},
it was convenient to write
\[ 
d_n \, 
\bar \rho_n\left (c_n + d_n s \right)
=
d_n 
\sum_{k=0}^{n-1}
b_k^{(n)} (c_n + d_n s)^{2(k-\kappa_n 
)+1}
e^{-2 \kappa_n (V(c_n + d_n s) - \log(c_n + d_n s))}.\]
If we integrate
\eqref{eq:IneqIntegrandStrongly} in $t$
we may see that the term
$d_n 
\sum_{k=0}^{n-1}
b_k^{(n)} (c_n + d_n s)^{2(k-\kappa_n 
)+1}$ can be bounded for
$1-\widetilde \varepsilon <c_n + d_n s < 1 + 
\widetilde \varepsilon$
by a constant. Meanwhile,
by taking $\varepsilon>0$ from condition
\eqref{eq:2PotentialEpsilon} and choosing any
$\gamma>1/\varepsilon$,
we can find $C>0$ such that term
$e^{-2 \kappa_n (V(c_n + d_n s) - \log(c_n + d_n s))}$
is bounded 
by a multiple of
$e^{-Cs}$ 
in the region where $s \geq x$ and
$c_n + d_n s \leq 1+1/(\log n)^{\gamma}$.
This completes the study of the maxima.

The convergence of
the point process of moduli is
explained at the end of Subsection 
\ref{sub:IdeaOfProofGumbel}.
\end{proof}

\begin{proof}[Proof of Theorem \ref{th:WeaklyKappa}]
It follows
the same steps as the proof of
Theorem \ref{th:Strongly} where now we define
$c_n = 1 + \frac{\Delta_n}{2(\kappa_n - n)}$
and $d_n = \frac{1}{2(\kappa_n - n)}$. 
Lemma \ref{lemG:LogPotentialUniformBound}
justifies that we only need to prove
\[\lim_{n \to \infty} \int_{x}^\infty 
d_n \, 
\bar \rho_n\left (c_n + d_n s \right)
\mathrm d s =
\exp\left(-(1+A)^{-1} e^{- x }\right) \]
for every $x \in \mathbb R$. Lemma \ref{lemG:LogKernelNear} allows us to
focus on $c_n + d_n s 
\leq 1 + \frac{1}{\sqrt{\kappa_n -n}}$
 while \eqref{eq:IneqIntegrandWeakly} together
with  $V(r) = \log r$ for $r \geq 1$
allows us to bound
$d_n \, 
\bar \rho_n\left (c_n + d_n s \right)$
by a multiple of $e^{-s}$
in the region where $s \geq x$ and
$c_n + d_n s 
\leq 1 + \frac{1}{\sqrt{\kappa_n -n}}$.
The proof may be completed by
using Proposition \ref{prop:IntensityWeakly}.
\end{proof}

\section{Proofs of the exponential behavior}
\label{sec:ProofExponential}

We begin by the proof of 
Theorem \ref{th:HardEdge}, which is motivated
by \cite{Seo}.
It follows
almost the same steps as
the one of Theorem \ref{th:Strongly}
and Theorem \ref{th:WeaklyKappa} with the main difference
that we can 
now use the behavior of the kernel at the edge
stated in Proposition
\ref{prop:HardEdgeKernelLimit}.

\begin{proof}[Proof of Theorem \ref{th:HardEdge}]

For this case, the 
convergence
of the point process of moduli implies the
fluctuations of the maxima.
Nevertheless, for simplicity,
we shall follow the steps explained
in Subsection \ref{sub:IdeaOfProofGumbel}
and show first the fluctuations of the maxima.
Let us choose $m_n = 1-\frac{x}{n^2}$ and recall
that
\[
\bar \rho_n(r) = 
\sum_{k=0}^{n-1}
b_k^{(n)} r^{2k+1} e^{-2\kappa_n V(r)}
\quad \mbox{ where } \quad
b_k^{(n)} 
=\left(\int_{0}^1
s^{2k+1}e^{-2\kappa_n V(s)}
\mathrm d s
\right)^{-1}.\]
To see that
\eqref{eq:UniformConvCondition} holds
we notice that there exists $C<1$ such that
\[b_k^{(n)}
\leq
\left(\int_{1-\varepsilon}^1 
s^{2k+1} e^{-2\kappa_n C \log s} 
\mathrm d s \right)^{-1} \leq 
\left(\int_{1-\varepsilon}^1 
s^{2n-1} e^{-2\kappa_n C \log s} 
\mathrm d s \right)^{-1}\]
which is bounded by $\widetilde C n$ for some
$\widetilde C > 0$ independent of $n$ and $k$.
Then, we have
\[\sup_{k \in \{0,\dots,n-1\}}
b_k^{(n)}\int_{1-\frac{x}{n^2}}^1
r^{2k+1} e^{-2\kappa_n V(r)}\mathrm d r
\leq \widetilde C n 
\int_{1-\frac{x}{n^2}}^1
r e^{-2(1-Q) \log r}\mathrm d r 
\xrightarrow[n \to \infty]{}0.\]
Define
\[f_n(r)=\frac{1}{n^2}\sum_{k=0}^{n-1} b_k^{(n)}
\left(1+\frac{r}{n}\right)^{2k}e^{-2\kappa_n 
V\left(1+\frac{r}{n}\right)}
\quad \mbox{ and } \quad f(r)=
2\int_0^Q
(q-t) e^{2r(q-t)} \mathrm d t.\]
Notice that $2\kappa_n V\left(1+\frac{r}{n}\right)
\to 2r(1-q)$ uniformly on compact sets
of $(-\infty,0]$. Then,
we may use Proposition \ref{prop:HardEdgeKernelLimit}
and recall that
$2\pi a_k^{(n)} = b_k^{(n)}$
to obtain that
\[f_n(r) \xrightarrow[n \to \infty]{} f(r)\]
for $r$ uniformly on compact sets of 
$(-\infty,0]$.
By the steps
described in Subsection \ref{sub:IdeaOfProofGumbel},
we are interested in the limit of
\[
\frac{1}{n^2}\bar \rho_n\left(1+\frac{s}{n^2} \right) = \left(1+\frac{s}{n^2}
\right)f_n\left(\frac{s}{n}\right).\]
But since the convergence of $f_n$ towards
$f$ is uniform in compact sets, we have that
\[\frac{1}{n^2}\bar \rho_n\left(1+\frac{s}{n^2} \right) 
\xrightarrow[n \to \infty]{} 
f(0),\]
for $s$ uniformly on compact sets of $\mathbb R$.
Then
\[\int_{x}^0
\frac{1}{n^2}\bar \rho_n\left(1+\frac{s}{n^2} \right)
\mathrm d s 
\xrightarrow[n \to \infty]{} 
-x f(0) = -2 x  \int_0^Q(q-t) \mathrm d t
=x \left(q^2 - (q-Q)^2\right),\]
which is what we wanted to prove. 
For the point process convergence, we may
proceed as explained 
at the end of Subsection \ref{sub:IdeaOfProofGumbel}.
\end{proof}
Now, we proceed to the proof of
Proposition
\ref{prop:HardEdgeKernelLimit}.
\begin{proof}[Proof of Proposition
\ref{prop:HardEdgeKernelLimit}]
Define
\[
F_n(z) = \frac{1}{n^2}
\sum_{k=0}^{n-1}
a_k^{(n)}
\left(1+\frac{z}{n}\right)^k.\]
Then, we have to prove that
\[\lim_{n \to \infty}
F_n\left(z+\bar w + \frac{z \bar w}{n} \right)
=\int_0^Q
\frac{q-t}{\pi} e^{-t(z+\bar w)} \mathrm d t
\]
for $(z,w)$ uniformly on compact
sets of $\mathbb C \times \mathbb C$.
Since $|F_n(z)| \leq F_n(r)$
whenever $|z| \leq r$, by Montel's
theorem it is enough to show that
\[
\lim_{n \to \infty} F_n(r)
= 
\int_0^Q
\frac{q-t}{\pi}  e^{-r t} 
\mathrm d t
\]
for every $r \geq 0$.
Define
$t^*= \frac{\lfloor tn \rfloor}{n}$
so that
\[F_n(r)
= \int_0^1 \frac{a^{(n)}_{n t^*}}{n}
\left(1+\frac{r}{n} \right)^{n t^* }
\mathrm d t.\]
We will study
the limiting behavior of
\[\left(
\frac{a^{(n)}_{n t^*}}{n} \right)^{-1}
= 
2\pi n \int_{0}^1
s^{2 n t^*+1} e^{-2\kappa_n V(s)} 
\mathrm d
s
=
n \int_{0}^\infty
e^{-2\kappa_n V(s)
+ 2 n t^* \log s} 
s \mathrm d
s \]
and, due to Laplace's method,
it will depend
on the points
where
$V(s) - t \log s$
is minimum. 
We know that the minimum is negative 
if $t < 1-Q$ so that
the sequence
converges exponentially to infinity.
On the other hand, 
if $t > 1 - Q$,
the minimum is zero and it
is attained at $s = 1$
so that
\begin{align*}
\left(
\frac{a^{(n)}_{n t^*}}{n} \right)^{-1}
&=
2\pi\int_{-n }^0
\left(1+\frac{s}{n} \right)^{2n t^* + 1}
e^{-2 \kappa_n V
\left(1+\frac{s}{n} \right)}
\mathrm d s								\\
&\xrightarrow[n \to \infty]{}	  			
2\pi\int_{-\infty}^0
e^{2ts} e^{-2(1-q)s} \mathrm d s			\\
&=\frac{\pi}{t - (1-q)}.
\end{align*}
By bounding
$V(s)$ from above by a
multiple of 
$\log s$ we can see that
 there is $C>0$ such that
\[\frac{a^{(n)}_{n t^*}}{n}
\left(1+\frac{r}{n} \right)^{n t^* }
1_{[0,1]}(t)
\leq C(t+C) e^{rt}1_{[0,1]}(t)\]
so that we may apply Lebesgue's dominated convergence
theorem and obtain
\[F_n(r) \xrightarrow[n \to \infty]{}
\int_{1-Q}^1
\frac{t - (1-q)}{\pi}
e^{r t} \mathrm d t
=\int_0^Q \frac{q-s}{\pi}
e^{r(1-s)} \mathrm d s.\]

\end{proof}

\section{Appendices}
\label{sec:Appendices}

\subsection{Random number of particles}

\label{subAp:RandomNumber}

For completeness, we
state here the
complementary version
of Theorem \ref{th:MaxFiniteParticles}.
Here the precise limiting behavior will
depend on some other parameters
which make the statement a little
more cumbersome.

\begin{theorem}[Finite limiting process: Random number of particles]
\label{th:MaxFiniteParticlesRandom}

Suppose that, for some
positive numbers
$\alpha, \gamma, Q_{\scaleto{+}{4.5pt}}, 
Q_{\scaleto{-}{4.5pt}}> \nolinebreak 0$,
\begin{itemize}
\item 
	\itemEqNoLabel{V(r) > \log r \,
	\mbox{ for every } 
	r > 1,}
\item
	\itemEqNoLabel{\lim_{r \to \infty}
 	r^\alpha \left( V(r) - \log r \right) = 
 	\gamma,}
\item 
	\itemEqNoLabel{\lim_{r \to 1^+} \frac{V(r)}{r-1}
	 =1+ Q_{\scaleto{+}{4.5pt}} \mbox{ and }}
\item
	\itemEqNoLabel{ \lim_{r \to 1^-} \frac{V(r)}{r-1}
	 =1- Q_{\scaleto{-}{4.5pt}}.}
\end{itemize}
Suppose that, for some 
$\chi \in (0,\alpha/2] $ 
such that
$\alpha/2-\chi \in \mathbb Z$ and
for some
$\xi \in \mathbb R$,
\[\kappa_n = n + \chi + \frac{\xi}{\log n}
+ o \left(\frac{1}{\log n} \right). \]
Consider a sequence
$\{Y_k\}_{k \geq 0}$
of independent random variables
where $Y_k$ follows the law
\[
\frac{r^{-2(k + \chi)}
e^{-r^{-\alpha}}
1_{(0,\infty)}(r)}
{\alpha
 \Gamma\left(\frac{2(k+ \chi)}{\alpha}
\right)}
\frac{\mathrm d r}{r}\]
and define
\[p=
\frac{\alpha \gamma
e^{2\xi/\alpha}
(Q_{\scaleto{+}{4.5pt}}+
Q_{\scaleto{-}{4.5pt}})}{
Q_{\scaleto{+}{4.5pt}}
Q_{\scaleto{-}{4.5pt}}}\]
Consider a
random variable 
$N$ independent of
$\{Y_k\}_{k \geq 0}$
such that
\[\mathbb 
 P
\left(N = \frac{\alpha}{2} - \chi
\right)= 
\frac{p}{1+p}\quad
\mbox{and} 
\quad \mathbb 
 P
\left(N = \frac{\alpha}{2} - 
\chi + 1
\right)= 
\frac{1}{1+p}.\]
Then,
for every continuous function
$f: [0,\infty) \to \mathbb R$
whose support is contained in
$(0,\infty)$,
 \[
 \sum_{k=1}^n f(n^{-1/\alpha}
 |x_k^{(n)}|)
\xrightarrow[n \to \infty]{\mathrm{law}} 
 \sum_{k=0}^{N-1}
 f\left( (2\gamma)^{1/\alpha}Y_k
 \right).\]
In particular, 
$n^{-1/\alpha}\max \{
|x_1^{(n)}|
,\dots, |x_n^{(n)}|\}$
converges in law to
the maximum of 
$\{Y_k\}_{0\leq k \leq N-1}$.
More explicitly, for every $t>0$,
				\[
				\lim_{n \to \infty}\mathbb P
	\left(n^{-1/\alpha}
	\max\{|x_1^{(n)}|,\dots,|x_n^{(n)}|\}
	\leq t
	\right)\]
	\[
	=\left(
				\prod_{k=0}^
				{ \alpha/2-\chi
				- 1
				}
	\frac{\Gamma
	 \left( 
	 \frac{2(k+\chi)}{\alpha},
	 2\gamma t^{-\alpha} \right)
	}
	{\Gamma
	 \left( 
	 \frac{2(k+\chi)}{\alpha} 
	\right)} 
	\right)
	\left(
	e^{-2\gamma t^{-\alpha}}
	+
	(1-e^{-2\gamma t^{-\alpha}})
	\frac{p}{1+p}
	\right).\]
 
\end{theorem}	

\begin{proof}
As in the proof of 
Theorem \ref{th:MaxFiniteParticles}, 
this theorem is a consequence of
Proposition \ref{propC:MinFiniteOne}
together
with 
\cite[Theorem 4.5.3]{HoughKrishnapurPeresVirag}.
We could have also used
\cite[Theorem 4.7.1]{HoughKrishnapurPeresVirag}
which is specific for
radial systems.
\end{proof}

Notice that in the extreme case
$\alpha = 2\chi$ the limiting 
distribution
of the maxima is a convex combination
of a Fréchet distribution and
a Dirac measure at zero. The same techniques can be used
to generalize this theorem 
to a case where
$V(r)$ has different kind of singularities
at $r=1$. In those generalizations actual Fréchet
distributions can be obtained and the
point process that accompany
the farthest particle would go
to a point process with just one particle.

\subsection{Comparison between the kernel and Gumbel fluctuations}

\label{subAp:KernelGumbel}

It is interesting to compare 
the results from Theorem \ref{th:Strongly}
and Theorem \ref{th:WeaklyKappa} with
the limiting behavior of the point process at 
a point $z$ of the unit circle, let us say $z=1$.
Since, as explained in 
\eqref{eqDet:ContinuityKernelProcess}, the
convergence of the point process
is implied by the convergence
of the kernels, we shall state now
the convergence of these kernels (up to a conjugation
of norm one). 
The proofs
are given at the end of this subsection. Define
\[K_n(z,w) =
\sum_{k=0}^{n-1}
a_k^{(n)}
z^{k} \bar w^k
e^{-\kappa_n V\left(\left|z
 \right|\right)}
 e^{-\kappa_n V\left(\left|w
\right| \right)}
\quad \mbox{ where } \quad 
a_k^{(n)}
= \left(2\pi\int_0^\infty
s^{2k + 1} e^{-2\kappa_n V(s)} 
\mathrm d s\right)^{-1}
\]
and let
\[\widetilde K_n(z,w) = 
\left(\frac{\bar z}{|z|} \right)^{\lfloor \kappa_n \rfloor} K_n(z,w)
\left(\frac{w}{|w|} \right)^{\lfloor \kappa_n \rfloor}. \]
We will state the versions for $\alpha>1$
since the versions for $\alpha=1$ are 
somewhat longer to state.

\begin{proposition}[Limiting kernel for strongly confining potentials]
\label{prop:StronglyKernelLimit}
Suppose $V$ satisfies
the \emph{standard properties} and
\begin{itemize}
\item 
	\itemEq{V(r) = \log r
	+ \frac{\lambda}{\alpha}
	(r-1)^\alpha
	+o(r-1)^{\alpha}
	\mbox{ as } r \to 1,
						\label{eqK:VTwoSidedBehavior}}
\item \itemEq{V(r) > \log r
\, \mbox{ for every }
r > 1 \mbox{ and } \nonumber}
\item 
   \itemEq{ \nonumber
	\liminf_{r \to \infty} 
	\frac{V(r)}{\log r} >1
	\quad \mbox{(strongly confining),}}
\end{itemize}
for some $\alpha > 1$
and $\lambda >0$. Suppose that
$\kappa_n/n \to 1$ and that
\[\lim_{n \to \infty}
\frac{n-\kappa_n }{n^{1/\alpha}}
=\zeta \in 
(-\infty,\infty].\]
Then, if $W: \mathbb C \to \mathbb R$
is defined by
$W(x+iy) = \frac{\lambda}{\alpha}
|x|^\alpha$, we have that
\[\lim_{n \to \infty} \frac{1}{n^{2/\alpha}}
\widetilde K_n\left(
1 + \frac{z}{n^{1/\alpha}},1 + 
\frac{w}{n^{1/\alpha}}\right) =
\frac{e^{-W(z)} e^{-W(w)}}{2\pi} \int_{-\infty}^{\zeta}
\frac{
e^{(z + \bar w) t} }
{ \int_{-\infty}^\infty
e^{2st}
e^{-2 W(s)} 
\mathrm d s	}
\mathrm d t\]
for $(z,w)$ uniformly on compact subsets of 
$\mathbb C \times \mathbb C$.

\end{proposition}

Notice that 
in condition \eqref{eqK:VTwoSidedBehavior}
there is no need of $\varepsilon$
as in
\eqref{eq:2PotentialEpsilon} from Theorem  
\ref{th:Strongly}.
When $\zeta = \infty$ we obtain the 
weighted Bergman kernel with weight
\[x + iy \mapsto e^{-2\frac{\lambda}{\alpha}
|x|^\alpha} \]
which is, for
$\alpha = 2$, 
the well-known Ginibre kernel up to a conjugation.
For $\zeta<\infty$, we obtain a truncated
version of a Bergman kernel. 
Moreover, for $\alpha = 2$ 
we get the well-known
error function kernel
if $\zeta = 0$ and a horizontally 
translated version of it
for general finite $\zeta$.
This error function kernel is known to be universal
under some analyticity conditions \cite{Aron}.

\begin{proposition}[Limiting kernel for weakly confining potentials]
\label{prop:WeaklyKernelLimit}
Suppose $V$ satisfies
the \emph{standard properties} and
\[
V(r) 
= \log r + \frac{\lambda}{\alpha}
(1-r)^{\alpha}
+ o(1-r)^{\alpha} \mbox{ as } r\to 1^{-}
\quad \mbox{ and }
\quad V(r) = \log r \, \mbox{ for } \, r \geq 1,\]
for some $\alpha > 1$
and $\lambda >0$. Suppose that
$\kappa_n>n$ satisfies
\[ \lim_{n \to \infty}
\frac{n-\kappa_n }{n^{1/\alpha}}
= \zeta \in 
(-\infty,0].\]
Then, if $W: \mathbb C \to \mathbb R$
is defined by
$W(x+iy) = \frac{\lambda}{\alpha}
|\min(x,0)|^\alpha$, we have that
\[\lim_{n \to \infty} \frac{1}{n^{2/\alpha}}
\widetilde K_n\left(
1 + \frac{z}{n^{1/\alpha}},1 + 
\frac{w}{n^{1/\alpha}}\right) =
\frac{e^{-W(z)}e^{- W(w)}}{2\pi} \int_{-\infty}^{\zeta}
\frac{
e^{(z + \bar w) t} }
{ \int_{-\infty}^\infty
e^{2st}
e^{-2 W(s)}
\mathrm d s	}
\mathrm d t\]
for $(z,w)$ uniformly on compact subsets of 
$\mathbb C \times \mathbb C$.


\end{proposition}

The version for $\alpha=1$, not stated
here, does not only depend on the behavior of
the potential at the unit circle.
It can be understood by using the same methods,
and results similar to
Proposition \ref{prop:HardEdgeKernelLimit}
would be obtained.
We may see another example in
\cite[Theorem 4.3]{GarciaZelada2}.

In Proposition \ref{prop:WeaklyKernelLimit},
when $\zeta = 0$,
we obtain
the weighted Bergman kernel
with weight
\[x + iy \mapsto e^{-2\frac{\lambda}{\alpha}
|\min(x,\,0)|^\alpha} .\]
For $\zeta<0$, we obtain a truncated
version of that Bergman kernel. Furthermore,
defining $W$ as in Proposition \ref{prop:WeaklyKernelLimit}, we have that
for $(z,w)$ uniformly on compact sets
of $\{x+iy \in \mathbb C: x > 0\}^2$,
\[
\frac{e^{-W(z)}e^{- W(w)}}{2\pi} \int_{-\infty}^{\zeta}
\frac{
e^{(z + \bar w) t} }
{ \int_{-\infty}^\infty
e^{2st}
e^{-2 W(s)}
\mathrm d s	}
\mathrm d t
\xrightarrow[\lambda \to \infty]{}
\int_{-\infty}^\zeta 
\frac{t e^{(z+\bar w)t}}{\pi} \mathrm d t\]
which is the unweighted Bergman kernel
of $\{x+iy \in \mathbb C: x > 0\}$
when $\zeta = 0$. 
The $\alpha=2$ version
\[K(x+iy,a+ib)=
\frac{
e^{-\frac{\lambda}{2}\min(x,0)^2}
e^{-\frac{\lambda}{2}\min(a,0)^2}}{2\pi} 
\int_{-\infty}^{\zeta}
\frac{
e^{(x+ a + i(y -b)) t} }
{ \int_{-\infty}^0
e^{2st}
e^{-\lambda s^2}
\mathrm d s	- \frac{1}{2t}}
\mathrm d t\]
has recently appeared in
\cite{Ameur} (at least for
$\zeta = 0$) and it satisfies
the nice but expected property
\[e^{i\lambda (T - x/2)y} K(x-T+iy,a-T+ib)  e^{-i\lambda (T - a/2)b}
\xrightarrow[T \to \infty]{} 
\frac{\lambda}{2\pi} e^{-\frac{\lambda}{4}\left( 
(x-a)^2 + (y-b)^2\right)}
e^{i\frac{\lambda}{2}(y a - xb)},\]
where we may recognize 
the Ginibre kernel at the right-hand side. In this way,
the $\alpha=2$ case for $\zeta = 0$ interpolates
between
the determinantal point process that has the
 Bergman kernel as its kernel and the 
Ginibre point process.
Now, we state the promised relation between
the Gumbel fluctuations
and the limiting kernels.

\begin{remark}[Connection between Gumbel fluctuations
and the limiting kernel]
\label{rem:GumbelAndEdge}
Choose $\alpha>1$ and $\lambda>0$ and take
$W:\mathbb R \to \mathbb R$ defined by
 $W(r) = \frac{\lambda}{\alpha}|r|^\alpha$.
Let us define $f: \mathbb R \to \mathbb R$
by
\begin{equation}
\label{eq:fInConnection}
f(r)=e^{-2W(r)} \int_{-\infty}^{0}
\frac{
e^{2r t} }
{ \int_{-\infty}^\infty
e^{2st}
e^{-2 W(s)} 
\mathrm d s	}
\mathrm d t.
\end{equation}
Using a change of variables in the integral,
we may see that
\[f(r) \sim_{r \to \infty} 
\frac{e^{-\frac{2\lambda}{\alpha}r^\alpha}}{2 r
\int_{-\infty}^\infty e^{-2W(s)} \mathrm d s}.\]
Then, the sought connection, which may be seen
as a connection between 
Proposition \ref{prop:StronglyKernelLimit} and
Proposition \ref{prop:IntensityStrongly}, is
the convergence
\[\left(\frac{\Delta_n}{2\lambda n}\right)^{1/\alpha} \frac{n^{2/\alpha}}{\Delta_n}f\left(n^{1/\alpha}
\left(\frac{
\Delta_n}{2\lambda n}\right)^{1/\alpha} 
\left(1+\frac{s}{\Delta_n} \right)\right)
\xrightarrow[ n \to \infty ]{} \frac{e^{-s}}{2
\int_{-\infty}^\infty e^{-2W(s)}\mathrm d s},
 \]
 where we recognize 
 $A$ from Theorem \ref{th:Strongly}
 as the denominator. 
 Nevertheless,
 the 
 parameter $\xi$ in Theorem 
 \ref{th:Strongly} can not be guessed
 by only using the behavior at the unit circle.
On the other hand, if we define
$W(r) = \frac{\lambda}{\alpha} |\min(r,0)|^\alpha$ 
 and $f$ as in \eqref{eq:fInConnection}, we notice
 that
 \[f(r) \sim_{r \to \infty}
 \frac{1}{2r^2}.
 \]
In particular, the behavior at the unit circle
would not allow us to guess
Theorem \ref{th:WeaklyKappa}. 

\end{remark}

Now, we proceed to give the proofs
of Proposition \ref{prop:StronglyKernelLimit} and Proposition \ref{prop:WeaklyKernelLimit}.
The ideas are somewhat simpler
than the ones used
to prove Proposition \ref{prop:IntensityStrongly}
and Proposition \ref{prop:IntensityWeakly}.

%

\begin{proof}[Proof of Proposition 
\ref{prop:StronglyKernelLimit}]

Define
\[F_n(z)=\frac{1}{n^{2/\alpha}}
\sum_{k=0}^{n-1}
a_k^{(n)}
\left(1 + 
\frac{z}{n^{1/\alpha}} \right)^{k - 
\lfloor \kappa_n \rfloor},
\quad \theta_n(z)
= \left|1+\frac{z}{n^{1/\alpha}}
\right|^{\lfloor \kappa_n\rfloor - \kappa_n}
\]
and
\[W_n(z) = 
\kappa_n \left(V\left(
\left|1+ \frac{z}{n^{1/\alpha}}
\right| \right) - 
\log \left|1+ \frac{z}{n^{1/\alpha}}
\right|\right) \]
so that we have
\[\frac{1}{n^{2/\alpha}}
\widetilde K_n\left(
1 + \frac{z}{n^{1/\alpha}},1 + 
\frac{w}{n^{1/\alpha}}\right) 
= F_n\left(z + \bar w + 
\frac{z \bar w}{n^{1/\alpha}} \right)
e^{-W_n(z)}
 e^{-W_n(w)}
 \theta_n(z)  \theta_n(w).\]
Then, it is enough to show that
$F_n$, $W_n $ and $\theta_n$ converge uniformly
on compact sets of $\mathbb C$.
Since $\lfloor \kappa_n \rfloor - \kappa_n$
is bounded we have that, for $z$ uniformly on
compact sets of $\mathbb C$,
\[\theta_n(z) \xrightarrow[n \to \infty]{} 1.\]
By \eqref{eqK:VTwoSidedBehavior} we have
that, for $x+iy$ uniformly on compact sets of
$\mathbb C$,
\[W_n(x+iy) \xrightarrow[n \to \infty]{} 
\frac{\lambda}{\alpha}
|x|^\alpha.\]
We only need to understand $F_n$.
But  for any $R>0$,
$\sup_{|z|\leq R}|F_n(z)| \leq F_n(R) + F_n(-R)$
for $n$ large enough. So, if
$F_n|_{\mathbb R}$ converges pointwise, by Montel's
theorem $F_n$ converges uniformly
on compact sets of $\mathbb C$. Now, take
$r \in \mathbb R$,
define
$t^{*} = 
\frac{\lfloor t n^{1/\alpha} \rfloor}
{n^{1/\alpha}} $
and perform the change of variables
$t^* n^{1/\alpha} = k - \lfloor \kappa_n \rfloor$
so that
\[F_n(r) = 
\bigints_{-\frac{\lfloor \kappa_n \rfloor}
{n^{1/\alpha}}}^{
\frac{n - \lfloor \kappa_n \rfloor}{n^{1/\alpha}}}
\frac{a_{\lfloor \kappa_n \rfloor
+t^* n^{1/\alpha}}^{(n)}}{n^{1/\alpha}}
\left(1 + \frac{r}{n^{1/\alpha}} 
\right)^{t^{*} n^{1/\alpha}}
\mathrm d t .\]
We have
\[
\lim_{n \to \infty}
\left(1 + \frac{r}{n^{1/\alpha}} 
\right)^{t^{*} n^{1/\alpha} }
= e^{r t}.
\]
For the term 
\begin{align}
\left(
\frac{a_{\lfloor \kappa_n \rfloor +
t^* n^{1/\alpha}}^{(n)}}{n^{1/\alpha}}
\right)^{-1}
&=
2\pi n^{1/\alpha}\int_0^\infty
s^{2(\lfloor \kappa_n \rfloor + 
t^{*}n^{1/\alpha}) + 1} e^{-2\kappa_n V(s)} 
\mathrm d s					\nonumber		\\
&=
2\pi n^{1/\alpha}\int_0^\infty
s^{2t^{*}n^{1/\alpha} + 2(\lfloor \kappa_n \rfloor
- \kappa_n) + 1} 
e^{-2\kappa_n \left(V(s) - \log s \right)} 
\mathrm d s					\nonumber		\\
&=
2\pi \int_{-n^{1/\alpha}}^\infty
\left(1 + \frac{s}{n^{1/\alpha}} \right)^{2t^{*}
 n^{1/\alpha} + 2 (\lfloor \kappa_n \rfloor
- \kappa_n) + 1} 
e^{-2\kappa_n \left(V
\left(1 + \frac{s}{n^{1/\alpha}}\right) 
- \log \left(1 + \frac{s}{n^{1/\alpha}} \right) \right)} 
\mathrm d s		\label{eq:ImportantIntegralForA}				
\end{align}
we can notice that it is enough to consider
the integral from 
$-\varepsilon n^{1/\alpha}$ to 
$\varepsilon n^{1/\alpha}$. Indeed,
since $V(s) > \log s$ if $s > 1$ 
and since
the potential is strongly confining,
we may see that
\[2\pi n^{1/\alpha}\int_{1+\varepsilon}^\infty
s^{2(\lfloor \kappa_n \rfloor + 
t^{*}n^{1/\alpha}) + 1} e^{-2\kappa_n V(s)} 
\mathrm d s		\xrightarrow[n \to \infty]{} 0\]
for every $\varepsilon>0$. On the other hand,
since $V(s) > \log s$ if $s<1$ and
since the potential is bounded near the origin,
we have that
 \[2\pi n^{1/\alpha}\int_{0}^{1-\varepsilon}
s^{2(\lfloor \kappa_n \rfloor + 
t^{*}n^{1/\alpha}) + 1} e^{-2\kappa_n V(s)} 
\mathrm d s		\xrightarrow[n \to \infty]{} 0\]
for every $\varepsilon>0$.
Then, for $\varepsilon>0$ small enough we may use
\eqref{eqK:VTwoSidedBehavior} to
control the integrand in
\eqref{eq:ImportantIntegralForA}
from $-\varepsilon n^{1/\alpha}$
to $\varepsilon n^{1/\alpha}$
and apply Lebesgue's dominated
convergence theorem to obtain
\[\left(
\frac{a_{\lfloor \kappa_n \rfloor +
t^* n^{1/\alpha}}^{(n)}}{n^{1/\alpha}}
\right)^{-1} \xrightarrow[n \to \infty]{}
2\pi \int_{-\infty}^\infty
e^{2st}
e^{-\frac{2\lambda}{\alpha} |s|^\alpha} 
\mathrm d s				.	\]
Since there exists $C>0$ and $\varepsilon>0$
such that if we define
\[E_C(x) = \left\{ 
\begin{array}{ll}
\exp(Cx)& \mbox{ if } x \leq 0 \\
\exp(C^{-1}x)& \mbox{ if } x > 0 
\end{array}\right. \]
we have
\[
\frac{a_{\lfloor \kappa_n \rfloor
+t^* n^{1/\alpha}}^{(n)}}{n^{1/\alpha}}
\left(1 + \frac{r}{n^{1/\alpha}} 
\right)^{2t^{*} n^{1/\alpha}}1_{
\left[-\frac{\lfloor \kappa_n \rfloor}
{n^{1/\alpha}},
\frac{n - \lfloor \kappa_n \rfloor}{n^{1/\alpha}}\right)}(s)
\leq C\left(\int_{-n^{1/\alpha}\varepsilon}^{n^{1/\alpha}\varepsilon}
E_C(st) e^{-C|s|^\alpha}
\mathrm d s
\right)^{-1} E_{C^{-1}}(rt) ,\]
where the right-hand side is integrable in $t$
for $n$ large enough,
we can apply Lebesgue's dominated convergence
theorem to conclude
\begin{align*}
\lim_{n \to \infty}
F_n(r) 
&=\frac{1}{2\pi} \int_{-\infty}^\eta
\frac{
e^{r t} }
{ \int_{-\infty}^\infty
e^{2st}
e^{-\frac{2\lambda}{\alpha} |s|^\alpha} 
\mathrm d s	}
\mathrm d t.
\end{align*}

\end{proof}

\begin{proof}[Proof of Proposition \ref{prop:WeaklyKernelLimit}]

As in the proof of Proposition
\ref{prop:StronglyKernelLimit},
we only need to understand 
\[F_n(z)=\frac{1}{n^{2/\alpha}}
\sum_{k=0}^{n-1}
a_k^{(n)}
\left(1 + 
\frac{z}{n^{1/\alpha}} \right)^{k - 
\lfloor \kappa_n \rfloor}.
\]
As before, we only need to study
the pointwise convergence of $F_n|_{\mathbb R}$.
Take
$r \in \mathbb R$ and
define
$t^{*} = 
\frac{\lfloor t n^{1/\alpha} \rfloor}
{n^{1/\alpha}} $
so that
\[
F_n(r) = 
\bigints_{-\frac{\lfloor \kappa_n \rfloor}
{n^{1/\alpha}}}^{
\frac{n - \lfloor \kappa_n \rfloor}{n^{1/\alpha}}}
\frac{a_{\lfloor \kappa_n \rfloor
+t^* n^{1/\alpha}}^{(n)}}{n^{1/\alpha}}
\left(1 + \frac{r}{n^{1/\alpha}} 
\right)^{t^{*} n^{1/\alpha}}
\mathrm d t .
\]
We have
\[
\lim_{n \to \infty}
\left(1 + \frac{r}{n^{1/\alpha}} 
\right)^{t^{*} n^{1/\alpha} }
= e^{r t}.
\]
The term 
\[
\left(
\frac{a_{\lfloor \kappa_n \rfloor +
t^* n^{1/\alpha}}^{(n)}}{n^{1/\alpha}}
\right)^{-1}
=
2\pi n^{1/\alpha}\int_0^\infty
s^{2(\lfloor \kappa_n \rfloor + 
t^{*}n^{1/\alpha}) + 1} e^{-2\kappa_n V(s)} 
\mathrm d s											
\]
can be split in two parts. The integral
from $1$ to $\infty$ can be explicitly
calculated and we obtain
\begin{equation}
\label{eq:From0To1Weakly}
2\pi n^{1/\alpha}\int_1^\infty
s^{2(\lfloor \kappa_n \rfloor + 
t^{*}n^{1/\alpha}) + 1} e^{-2\kappa_n V(s)} 
\mathrm d s				
\xrightarrow[n \to \infty]{}
-\frac{\pi}{t}
= 2\pi \int_{0}^\infty e^{2st}\mathrm d s
\end{equation}
For the integral from $0$ to $1$ there is no
difference from
the proof of Proposition \ref{prop:StronglyKernelLimit} and we get
\[2\pi n^{1/\alpha}\int_0^1
s^{2(\lfloor \kappa_n \rfloor + 
t^{*}n^{1/\alpha}) + 1} e^{-2\kappa_n V(s)} 
\mathrm d s			
\xrightarrow[n \to \infty]{}
2\pi \int_{-\infty}^0
e^{2st}
e^{-\frac{2\lambda}{\alpha} |s|^\alpha} 
\mathrm d s.
\]
In summary, recalling that
$W(r) = \frac{\lambda}{\alpha}
|\min(r,0)|^\alpha$, we have obtained
\[\left(
\frac{a_{\lfloor \kappa_n \rfloor +
t^* n^{1/\alpha}}^{(n)}}{n^{1/\alpha}}
\right)^{-1} \xrightarrow[n \to \infty]{}
2\pi \int_{-\infty}^\infty
e^{2st}
e^{-2W(s)} 
\mathrm d s				.	\]
Since this $F_n$ is dominated by
the $F_n$ from 
the proof of Proposition 
\ref{prop:StronglyKernelLimit},
we may apply Lebesgue's dominated convergence
theorem to obtain
\[
\lim_{n \to \infty}
F_n(r) 
=\frac{1}{2\pi} \int_{-\infty}^\eta
\frac{
e^{r t} }
{ \int_{-\infty}^\infty
e^{2st}
e^{-2W(s)} 
\mathrm d s	}
\mathrm d t.\]
We could have also used the inverse of the
left-hand side
of
\eqref{eq:From0To1Weakly} to bound the integrand
in $F_n(r)$ since it can be explicitly calculated.
\end{proof}

\subsection{Exponential fluctuations for different behaviors}

\label{subAp:Hard-edgeAlphaCase}

Recall that
\[
\bar \rho_n(r) = 
\sum_{k=0}^{n-1}
b_k^{(n)} r^{2k+1} e^{-2\kappa_n V(r)}
\quad \mbox{ where } \quad
b_k^{(n)} 
=\left(\int_{0}^1
s^{2k+1}e^{-2\kappa_n V(s)}
\mathrm d s
\right)^{-1}.\]
The $\alpha=2$ case of the following
theorem is treated in
\cite{Seo} with additional conditions on $V$.

In the following theorem, $V$ will be a continuous
function on $[0,1]$.

\begin{theorem}

\label{th:HardEdgeAlpha}

Suppose $V$ satisfies the \emph{standard properties} and
\[
V(r) 
= \log r + \frac{\lambda}{\alpha}
(1-r)^{\alpha}
+ o(1-r)^{\alpha} \mbox{ as } r\to 1^{-},\]
for some $\alpha > 1$
and $\lambda >0$.
Suppose that
$\kappa_n$ satisfies
\[ \lim_{n \to \infty}
\frac{n-\kappa_n }{n^{1/\alpha}}
= \zeta \in 
(-\infty,\infty].\]
Define $F:(-\infty,0] \to \mathbb R$ by
\[F(r) = e^{-2\frac{\lambda}{\alpha} |r|^\alpha}
\int_{-\infty}^{\zeta}
\frac{e^{2tr}}
{\int_{-\infty}^0 e^{2s t} e^{-2\frac{\lambda}{\alpha}
|s|^\alpha} \mathrm d s}\mathrm d t.\]
Then, for $r$ uniformly on compact
sets of $(-\infty,0]$,
\[
\frac{1}{n^{2/\alpha}}
\rho_n\left(1+\frac{r}{n^{1/\alpha}}\right)
\xrightarrow[n \to \infty]{}
F(r).
\]
Furthermore,
as point processes on $[0,\infty)$,
\[
\left\{
n^{2/\alpha}
(1 - |x_i^{(n)}|) : 1 \leq i \leq n
\right\}
\xrightarrow[n \to \infty]{\mathrm{law}}
\mathcal P,
\]
where $\mathcal P$ is a Poisson point process
on $[0,\infty)$ with intensity $F(0)$.
In particular, for every $t \geq 0$,
\[\mathbb P\left(n^{2/\alpha} \left(1-
M_n
\right) \leq t\right) \to 
1 - 
e^{-F(0)t}.\]

\end{theorem}

\begin{proof}

Taking $t^* = \frac{\lfloor t n^{1/\alpha} \rfloor}
{n^{1/\alpha}}$, we may write
\begin{align*}
&\frac{1}{n^{2/\alpha}}
\rho_n\left(1+\frac{r}{n^{1/\alpha}}\right)
=
\frac{1}{n^{2/\alpha}}\sum_{k=0}^{n-1}
b_k^{(n)}
\left(1+\frac{r}{n^{1/\alpha}}\right)^{2(k-
\lfloor \kappa_n \rfloor)+1}
e^{-2\kappa_n\left(V
\left(1+\frac{r}{n^{1/\alpha}} \right)
-\log\left(1+\frac{r}{n^{1/\alpha}}\right) \right)}	\\
&\hspace{25mm}=
e^{-2\kappa_n\left(V
\left(1+\frac{r}{n^{1/\alpha}} \right)
-\log\left(1+\frac{r}{n^{1/\alpha}}\right) \right)}
\bigint_{-\frac{\lfloor \kappa_n \rfloor}{n^{1/\alpha}}}^{\frac{n-\lfloor \kappa_n \rfloor}{n^{1/\alpha}}}
\hspace{-4mm} 
\frac{b_{\lfloor \kappa_n \rfloor + t^* n^{1/\alpha}}}
{n^{1/\alpha}}
\left(1+\frac{r}{n^{1/\alpha}}\right)^{2t^*n^{1/\alpha}
+1}
\mathrm d t .
\end{align*}
By noticing
that
\begin{align*}
\left(\frac{b_{\lfloor \kappa_n \rfloor + t^* n^{1/\alpha}}}
{n^{1/\alpha}}\right)^{-1}
&= \int_{-n^{1/\alpha}}^0
\left(1+\frac{s}{n^{1/\alpha}}
\right)^{2(\lfloor \kappa_n \rfloor - \kappa_n) +
t^* n^{1/\alpha} + 1} e^{-2\kappa_n
\left(V\left(1 + \frac{s}{n^{1/\alpha}}\right)
-\log\left(1 + \frac{s}{n^{1/\alpha}}\right)\right)}
\mathrm d s											\\
&\xrightarrow[n \to \infty]{}
\int_{-\infty}^0 e^{2s t} e^{-2\frac{\lambda}{\alpha}
|s|^\alpha} \mathrm d s
\end{align*}
we obtain 
\[
\frac{1}{n^{2/\alpha}}
\rho_n\left(1+\frac{r}{n^{1/\alpha}}\right)
\xrightarrow[n \to \infty]{}
e^{-2\frac{\lambda}{\alpha} |r|^\alpha}
\int_{-\infty}^{\zeta}
\frac{e^{2tr}}
{\int_{-\infty}^0 e^{2s t} e^{-2\frac{\lambda}{\alpha}
|s|^\alpha} \mathrm d s}\mathrm d t = F(r)
\]
for $r$ uniformly on compact
sets of $(-\infty,0]$. 
The justifications to apply
Lebesgue's dominated converges
theorem in the previous two integrals
are as in the proof of 
Proposition \ref{prop:StronglyKernelLimit}.
Finally, 
\[\frac{1}{n^{2/\alpha}}
\rho_n\left(1+\frac{r}{n^{2/\alpha}}\right)
\xrightarrow[n \to \infty]{} F(0)\]
also for $r$ uniformly on compact
sets of $(-\infty,0]$ so that
we obtain the result by using the ideas described
in Subsection \ref{sub:IdeaOfProofGumbel}
as in the proof of Theorem \ref{th:HardEdge}.

\end{proof}

\subsection{Gumbel distribution and weakly confining fluctuations}

\label{subAp:GumbelConnectionWeakly}

Here we establish a connection
between the limit 
of the maxima in 
Theorem \ref{th:MaxAnnulus} for $\widetilde R 
= \infty$
and the Gumbel distribution.

\begin{proposition}[Weakly confining
and Gumbel distribution]
\label{prop:GumbelWeakly}

For each $\chi > 0$ let $X_\chi$
be a random variable with
cumulative distribution function
	$$\mathbb P(X_\chi \leq t )
	=
		\prod_{k=0}^\infty 
	\left(1 - t^{-2k-2\chi} \right)$$
for every $t \geq 1$ and let
$\varepsilon_\chi > 0$ denote
the unique
solution to
	$$e^{\chi \varepsilon_{\chi}} 
	\varepsilon_{\chi}=1.$$
Then, as $\chi \to \infty$, we have
	$$ 
	2 \chi (X_\chi - 1 - \varepsilon_{\chi}/2)
	\to G$$	
where $G$ has a standard Gumbel distribution,
i.e.,
	$$\mathbb P(G \leq a)
	= e^{-e^{-a}}$$
for every $a \in \mathbb R$.

\begin{proof}

Define $b_{\chi}$ by
$$2b_{\chi} = \varepsilon_{\chi}+
a/\chi.$$
We have to prove that
	$$ \lim_{\chi \to \infty}
			\prod_{k=0}^\infty 
	\left(1 - (1+b_{\chi})^{-2k-2\chi} \right)
	=
	e^{-e^{-a}}.$$
We notice that,
as $\chi \to \infty$,
	$$\log\prod_{k=0}^\infty 
	\left(1 - (1+b_{\chi})^{-2k-2\chi} \right)
	\sim
	-\sum_{k=0}^\infty
	(1+b_{\chi})^{-2k-2\chi} 
	$$
which is due to
the fact that
$(1+b_\chi)^\chi \to \infty$ and
$\log(1-x) \sim -x + o(x)$.
Then
	$$\sum_{k=0}^\infty
	(1+b_{\chi})^{-2k-2\chi}
	=\frac{(1+b_\chi)^{-2\chi}}
	{1-(1+b_\chi)^{-2}}
	\sim 
	\frac{e^{-2\chi b_\chi + \chi O(b_\chi)^2}}
	{2b_{\chi}}
	\sim
	\frac{e^{-\chi \varepsilon_\chi}}
	{\varepsilon_{\chi}}
	e^{-a}=e^{-a}
	   $$
which concludes the proof.

\end{proof}
\end{proposition}

Similarly to Theorem \ref{th:Strongly}
and Theorem \ref{th:WeaklyKappa}, 
the proposition still holds if 
we change $\varepsilon_\chi$ to $\varepsilon_\chi = \chi^{-1}
\left(\log \chi - \log \log \chi \right)$.
Nevertheless, we preferred to use the defining relation
$e^{\chi \varepsilon}\varepsilon = 1$
due to its importance in the proof.

 \noindent
\small{ 
 \textit{E-mail address}: \texttt{david.garcia-zelada@univ-amu.fr}}

\end{document}